\documentclass[reqno]{amsart}
\usepackage{amssymb,amscd,verbatim, amsthm,graphics, color,latexsym,amsmath,multicol}
\usepackage{fancybox}
\usepackage{longtable}

\usepackage[all]{xy}

\usepackage[top=2.7cm, bottom=2.7cm, left=2.7cm, right=2.7cm]{geometry}

\newcommand \fk[1]{{{\mathfrak #1}}}
\newcommand \C[1]{{\mathcal #1}}

\newcommand \wti[1]{{\widetilde {#1}}}

\newcommand\fg{\mathfrak g}

\newcommand \bC{{\mathbb C}}

\newcommand \bH{{\mathbb H}}
\newcommand \bR{{\mathbb R}}
\newcommand \bZ{{\mathbb Z}}

\newcommand \bN{{\mathbb N}}

\newcommand \bbq{{\mathsf {q}}}

\newcommand\CA{{\C A}}

\newcommand\CC{{\C C}}
\newcommand\CH{{\C H}}

\newcommand\CO{{\C O}}

\newcommand\CR{{\C R}}
\newcommand\CP{{\C P}}
\newcommand\CG{{\C G}}

\newtheorem{theorem}{Theorem}[section]

\newtheorem{corollary}[theorem]{Corollary}
\newtheorem{lemma}[theorem]{Lemma}
\newtheorem{proposition}[theorem]{Proposition}
\newtheorem{definition}[theorem]{Definition}
\newtheorem{remark}[theorem]{Remark}
\newtheorem{example}[theorem]{Example}

\newcommand\End{\operatorname{End}}

\newcommand\Ind{\operatorname{Ind}}

\newcommand\fgen{{\mathsf{fg}}}

\newcommand\Id{\operatorname{Id}}

\def\<{\langle}
\def\>{\rangle}

\newcommand\R{{\mathcal R}}

\newcommand\so{{\mathcal O}}

\newcommand\D{{\mathcal D}}

\newcommand\sso{\mathcal O}

\newcommand\B{{\mathcal B}}
\newcommand\A{{\mathcal A}}
\newcommand\sA{{\mathcal A}}

\numberwithin{equation}{section}

\begin{document}

\title[On the reducibility of induced representations]{On the reducibility of induced representations for classical $p$-adic groups and related affine Hecke algebras}

\author
{Dan Ciubotaru}
        \address[D. Ciubotaru]{Mathematical Institute, University of Oxford, Oxford OX2 6GG, UK}
        \email{dan.ciubotaru@maths.ox.ac.uk}

 \author
 {Volker Heiermann}
 \address[V.Heiermann]{Aix-Marseille Universit\'e, CNRS, Centrale Marseille, I2M, UMR 7373, 13453 Marseille, France}
 \email{volker.heiermann@univ-amu.fr}

\begin{abstract}
 Let $\pi $ be an irreducible smooth complex representation of a general linear $p$-adic group and let $\sigma $ be an irreducible complex supercuspidal representation of a classical $p$-adic group of a given type, so that $\pi\otimes\sigma $ is a representation of a standard Levi subgroup of a $p$-adic classical group of higher rank. We show that the reducibility of the representation of the appropriate $p$-adic classical group obtained by (normalized) parabolic induction from $\pi\otimes\sigma $ does not depend on $\sigma $, if $\sigma $ is "separated" from the supercuspidal support of $\pi $. (Here, ``separated" means that, for each factor $\rho $ of a representation in the supercuspidal support of $\pi $, the representation parabolically induced from $\rho\otimes\sigma $ is irreducible.) This was conjectured by E. Lapid and M. Tadi\'c. (In addition, they proved, using results of C. Jantzen, that this induced representation is always reducible if the supercuspidal support is not separated.)

More generally, we study, for a given set $I$ of inertial orbits of supercuspidal representations of $p$-adic general linear groups, the category $\CC _{I,\sigma }$ of smooth complex finitely generated representations of classical $p$-adic groups of fixed type, but arbitrary rank, and supercuspidal support given by $\sigma $ and $I$, show that this category is equivalent to a category of finitely generated right modules over a direct sum of tensor products of extended affine Hecke algebras of type $A$, $B$ and $D$ and establish functoriality properties, relating categories with disjoint $I$'s. In this way, we extend results of C. Jantzen who proved a bijection between irreducible representations corresponding to these categories.
The proof of the above reducibility result is then based on Hecke algebra arguments, using Kato's exotic geometry.

\end{abstract}

\maketitle

\section{Introduction}

Let $G_d$ be the group of rational points of a classical group of relative semi-simple rank $d$ defined over a non-Archimedean local field $F$.\footnote{We don't make any assumption on the characteristic of $F$. Although in \cite{H3} characteristic $0$ is assumed, the relevant results used here from \cite{H3}, in particular \cite[Annex A, B, C]{H3}, do not use it.
(The restriction to characteristic $0$  has been made in \cite{H3} only because of the use of the results of C. M\oe glin, based on methods of J. Arthur, especially on the Langlands parameters for supercuspidal representations and reducibility points.) In \cite{H1}, no assumption on the characteristic of $F$ is made.} We understand by that that $G_d$ is either a general linear group, a symplectic group, a special odd orthogonal group, a full even orthogonal group or a unitary group, meaning that the latter ones are defined respectively by an alternating, symmetric or hermitian bilinear form. Remark that all these groups, except the even orthogonal one, are connected. The non-negative integer $d$ is  the rank of a maximal split torus of $G_d$. (In particular, $G_0$ is anisotropic.)

The type of $G_d$ and the splitting field will be fixed in the sequel, so that $G_{d'}$ will be a factor of a Levi subgroup of $G_d$ if $d'\leq d$. More generally, the Levi subgroups of $G_d$ will have the form $M\times G_{d'}$ (when $G_{d'}=1$, this is identified with $M$), where $M$ is a Levi subgroup of a general linear group with coefficients in a field $E$ which is equal to $F$ except if $G_d$ is a unitary group in which case $E$ equals the splitting field of $G_d$. (If $G$ denotes a general linear group, one can of course take $G_{d'}=1$.)

We will denote by $\R(G_d)$ the category of smooth complex representations of $G_d$.

Denote by $I$ a set of inertial orbits, closed under taking contragredients (or conjugate-duals if $E\ne F$), of (possibly different) general linear groups $GL_n(E)$, $n\in\Bbb N$. If $G$ is a classical group other than $GL_n$, fix an integer $d_0\geq 0$ and a supercuspidal representation $\sigma $ of $G_{d_0}$. If $G$ is a general linear group, then put $d_0=0$, so that in this case always $G_{d_0}=\{1\}$ and $\sigma =1$.

Let $\CC_{I;\sigma }$ be the direct sum of the full sub-categories of the category $\R (G_d)_{\fgen}$ of finitely generated representations of the groups $G_d$, $d\in\Bbb N$, whose irreducible subquotients contain each a representation of the form $\rho _1\otimes\cdots\otimes\rho _r\otimes\sigma $ in their supercuspidal support, with $\rho _i$ lying in some $\CO _i$ in $I$. (When $G_{d_0}=1$, then by convention $\rho _1\otimes\cdots\otimes\rho _r\otimes\sigma :=\rho _1\otimes\cdots\otimes\rho _r$.) If $G$ denotes a general linear group, write $\CC_I^{GL}$ instead of $\CC_{I;1 }$.

We first show (see Corollary \ref{c:2.6} with \ref{r:2.6}) that the category $\CC_{I;\sigma }$ is equivalent to the category of finitely generated right-modules over a direct sum of tensor products of endomorphism algebras of projective generators, which are known to be, possibly extended, affine Hecke algebras of type $A$, $B$ or $D$. It follows from this that for disjoint sets $I$ and $J$, one obtains a natural equivalence of categories $\Phi _{I,\sigma ;J,\sigma }:\CC_{I;\sigma }\otimes \CC_{J;\sigma }\rightarrow \CC_{I\cup J;\sigma }$, which is compatible with parabolic induction and Jacquet functor (Corollary \ref{c:2.7}).

On the other hand, parabolic induction defines a functor $M_{I,\sigma }:\CC_{I}^{GL}\otimes\CC_{I;\sigma }\rightarrow \CC_{I;\sigma }$. We show then that the equivalence of category $\Phi _{I,\sigma ;J,\sigma }$ is compatible with $M_{I,\sigma }$, $M_{J,\sigma }$, establishing some functorial equivalence (see Theorem \ref{t:2.10}). All this was motivated by the work of C. Jantzen \cite{J} (in particular \cite[Theorem 9.3]{J}, where he gives a bijection between irreducible objects of $\CC_{I;\sigma }\otimes \CC_{J;\sigma }$ and $\CC_{I\cup J;\sigma }$, satisfying certain properties that hold also in our situation). In fact, C. Jantzen considers ``real lines" instead of inertial orbits, which is more precise, but we recover also his result (cf. Remark 2.9 (i)).

Keep $\sigma $ fixed and consider now the set $S_{\sigma }$ of irreducible supercuspidal representations $\rho $ of general linear $p$-adic groups over $E$ such that the representation parabolically induced from $\sigma\otimes\rho $ (of the appropriate group $G_d$) is reducible. Denote by $\pi $ a complex smooth representation of a $p$-adic general linear group. We will say that $supp(\pi )\cap S_{\sigma }$ is empty or nonempty, if there are elements in $S_{\sigma }$ which are factors of representations in the supercuspidal support of $\pi $ or if there are no such representations in $S_{\sigma }$. (We leave it to the reader to give a precise definition to $supp(\pi )$, as we do not need it here.)

Using the above remarked relation with extended affine Hecke algebras of types $A$, $B$ and $D$, we prove by ``affine Hecke algebra''-arguments two results. Firstly, in section \ref{s:red}, we prove Theorem \ref{t:reducible2} which gives by Proposition \ref{p:reduc} the Hecke algebra counterpart of the  result of Lapid-Tadi\'c \cite[Theorem 1.1]{LT} that the representation induced by (normalized) parabolic induction from $\pi\otimes\sigma $ is always reducible if $supp(\pi)\cap S_{\sigma }\ne\emptyset$. This also implies the result of Lapid-Tadi\'c (cf. Remark \ref{r:reducibility}) (and conversely the result of Lapid-Tadi\'c implies this reducibility result for the appropriate extended affine Hecke algebras). Secondly, the main result of the paper is the verification of a conjecture by Lapid-Tadi\'c \cite{LT}, which says that the reducibility of $\pi\otimes\sigma $ does {\it not} depend on $\sigma $, whenever $supp(\pi)\cap S_{\sigma }=\emptyset$ (Theorem \ref{c:4.16}). To this end, we use Kato's construction \cite{Ka} of affine Hecke algebras of classical types using the exotic geometry, since this is particularly well adapted to the comparison of different specializations of the generic affine Hecke algebra.

\medskip

{\bf Acknowledgements.} D.C. was supported in part by the EPSRC grant EP/N033922/1 (2016).
V.H. has benefitted from a grant of Agence
Nationale de la Recherche with reference  ANR-13-BS01-0012 FERPLAY.
This paper was partly written while D.C. visited Aix-Marseille Universit\'e; he thanks the AMU Mathematics Institute for the hospitality and the support that made this work possible. Both authors participated in Spring 2017 in the special trimester on ``Representation Theory of Reductive Groups over local fields and applications to Automorphic forms" at the Weizmann Institute, Israel, where some of the ideas in this work were developed. The authors thank A. R. Aizenbud and D. Gourevitch for the invitation and support. They also thank E. Lapid for helpful remarks and for drawing their attention to \cite{LT} which motivated the material in section \ref{s:Bernstein}. Special thanks are due to the referee for several remarks that encouraged us to improve the results of the paper.

\section{Bernstein components and affine Hecke algebras}\label{s:Bernstein}

We keep the notation from the introduction.

\subsection{} In addition, we  denote by $\CO $ the union of the inertial orbit of a supercuspidal representation of a Levi subgroup of a general linear group with the inertial orbit of the (conjugate-)dual\footnote{This will be the contragredient representations, except if $G$ denotes a unitary group where one has to add conjugation by the non-trivial automorphism of $E/F$.}  of the supercuspidal representation, so that $\CO $ is stable under unramified twists and contragredients. We call $\CO $ a \it self-dual inertial orbit \rm and say that $\CO $ is homogeneous if it contains a representation of the form $\rho\otimes\cdots\otimes\rho$. The number of factors $\rho $ will then be called the length of $\CO $, denoted $l(\CO )$. We call the {\it support of }$\CO$, denoted $supp(\CO )$, the union of the inertial orbits of $\rho $ and of $\rho ^{\vee }$. Thus, the support is a self-dual inertial orbit of length $1$ in the homogeneous case. If $n$ is such that $\rho $ is a representation of $GL_n(E)$, then we write $n=n(\CO )$.

A set of supercuspidal representations of $GL_n(E)$, $n\geq 1$, stable under unramified twists and under conjugate-dual, can be uniquely written as a disjoint union of length $1$ self-dual inertial orbits.

Let $I$ be a set of length $1$ self-dual inertial orbits of (possibly different) general linear groups $GL_n(E)$, $n\in\Bbb N$, as in the introduction.
Let $\CO _1,\dots ,\CO_r$ be a finite set of homogeneous orbits with distinct support in $I$. We will put $\CO_1\otimes\cdots\otimes\CO _r:=\{\rho _1\otimes\cdots\otimes\rho _r\vert \rho_i\in\CO _i\}$ and $\CO_1\otimes\cdots\otimes\CO_r\otimes\sigma :=\{\rho_1\otimes\cdots\otimes\rho_r\otimes\sigma \vert\ \rho_i\in\CO_i \}$ (which will be the same as $\CO_1\otimes\cdots\otimes\CO _r$, when $G_{d_0}=1$). The sets $\CO:=\CO_1\otimes\cdots\otimes\CO _r$ and $\CO\otimes\sigma :=\CO_1\otimes\cdots\otimes\CO_r\otimes\sigma $ are a union of inertial orbits of a standard Levi subgroup of some $G_d$, $d\geq d_0$. If $G$ is not a general linear group, all of these inertial orbits give rise to the same Bernstein component of $\R (G_d)$. We denote it by $\R (G_d)_{\so\otimes\sigma }$. If $G$ is the general linear group, $\R (G_d)_{\so\otimes\sigma }$ will denote the finite union of the corresponding Bernstein components. Write $d=d(\CO\otimes\sigma )$. One has the formula \[d(\CO\otimes\sigma )=d_0+\sum _{i=1}^r n(\CO _i)l(\CO _i).\]

The following proposition is a direct consequence of the theory of the Bernstein center \cite{BD}:

\begin{proposition}\label{p:2.1}One has the equality
$$\CC_{I;\sigma }=\bigoplus _{r=1}^{\infty }\bigoplus_{\{\so _1,\dots ,\so _r\}} (\R (G_{d(\so_1\otimes\cdots\otimes\so_r\otimes\sigma )})_{\so_1\otimes\cdots\otimes\so_r\otimes\sigma })_{\fgen},$$
where the sum goes over finite sets of homogeneous orbits with distinct support in $I$ .
\end{proposition}

\subsection{} Let now $\CO $ be a tensor product of a finite set of homogeneous orbits $\CO _i$ with distinct support, $\CO :=\CO _1\otimes\cdots\otimes\CO _r$. In \cite{H1}, we fixed a certain projective generator $\CP _{\so ,\sigma }$ in $\R (G_d)_{\so\otimes\sigma }$  following Bernstein, see \cite{Ro}. We will write $\CP _{\so ,\sigma }^{G_d}$, if we want to underline the corresponding group. (If $G$ is a general linear group, $\CP _{\so }^{G_d}$ will be the direct sum of the projective generators corresponding to the different Bernstein components inside $\R (G_d)_{\so }$.)

In fact, if $M_{\so\otimes\sigma }$ denotes the Levi subgroup of $G_{d(\so\otimes\sigma )}$ on which the representations in $\CO\otimes\sigma$ are defined, then $\CP _{\so ,\sigma }$ is parabolically induced from a certain projective generator in $\R(M_{\so \otimes\sigma })_{\so\otimes\sigma }$. In particular, the projective generator $\CP _{\so ,\sigma }$ of  $\R (G_{d(\so\otimes\sigma )})_{\so\otimes\sigma }$ depends on the ordering of the $\CO _i$, while the Bernstein component $\R (G_{d(\so )\otimes\sigma })_{\so\otimes\sigma }$ doesn't. However, one has

\begin{proposition} If $\{\CO _1',\dots ,\CO _r'\}=\{\CO _1,\dots ,\CO _r\}$, then, with $\CO '=\CO _1'\otimes\cdots\otimes\CO _r'\otimes \sigma$, $P_{\so ',\sigma }$ and $P_{\so ,\sigma }$ are isomorphic.
\end{proposition}

\begin{proof}By \cite[1.10 and section 3.1]{H1}, the intertwining operator permuting the $\CO_i$ is an isomorphism.
\end{proof}

\subsection{} To fix notations, we choose in this subsection for each permutation $\CO '=\CO _1'\otimes\cdots\otimes\CO _r'$ of $\CO$ an isomorphism between the projective generators $P_{\so ',\sigma }$ and $P_{\so ,\sigma }$. It induces an isomorphism between the corresponding endomorphism algebras. In addition, we fix an ordering on the length 1 self-dual inertial orbits of the general linear groups. This induces an ordering on any finite set $\CO _1,\dots ,\CO _r$ of homogeneous inertial orbits with distinct support.

For each $i$, one deduces from this choice an embedding $\End_{G_{d(\so _i\otimes\sigma )}}(\CP _{\so _i,\sigma })\hookrightarrow \End_{G_{d(\so\otimes\sigma )}}(\CP_{\so ,\sigma })$: recall first that $\CP _{\so _i,\sigma }$ and $\CP _{\so ,\sigma }$ are defined through parabolic induction from suitable Levi subgroups. Now, if $i=r$,  the inductions are all w.r.t. standard parabolic subgroups so that transitivity holds, and the above embedding follows by functoriality of parabolic induction. For general $i$, to get an embedding, one forms first a permutation $\CO ^{(i)}$ of $\CO $ by permuting $\CO _i$ successively with the $\CO _j$, $j>i$. Then, one has an embedding $\End_{G_{d(\so _i\otimes\sigma )}}(\CP _{\so _i,\sigma })\hookrightarrow \End_{G_{d(\so\otimes\sigma )}}(\CP_{\so^{(i)} ,\sigma })$ given by parabolic induction as above. When conjugating this embedding by the isomorphism $\End_{G_{d(\so\otimes\sigma )}}(\CP_{\so  ,\sigma })\longrightarrow \End_{G_{d(\so\otimes\sigma )}}(\CP_{\so^{(i)} ,\sigma })$ fixed above, one gets our embedding.

\begin{proposition}\label{p:2.3}After the above identifications, one has a canonical isomorphism
$$\End_{G_{d(\so\otimes\sigma )}}(\CP_{\so ,\sigma })\simeq\bigotimes_{i=1}^r \End_{G_{d(\so _i\otimes\sigma )}}(\CP _{\so _i,\sigma }).$$
\end{proposition}

\begin{remark}
\begin{enumerate}
\item[(i)] This result  does not apply to the even special orthogonal group. (An erratum for a certain related statement in \cite{H1} has been made in \cite[end of appendix A]{H3}.)

\item[(ii)] For any permutation $\CO'=\CO_1'\otimes\cdots\otimes\CO _r'$, one has a canonical isomorphism of $\bigotimes_{i=1}^r \End_{G_{d(\so _i'\otimes\sigma )}}(\CP _{\so _i',\sigma })$ and        $\bigotimes_{i=1}^r \End_{G_{d(\so _i\otimes\sigma )}}(\CP _{\so _i,\sigma })$, whence an isomorphism with $\End_{G_{d(\so\otimes\sigma )}}(\CP_{\so ,\sigma })$. This is the isomorphism that will be chosen in the sequel.
\end{enumerate}
\end{remark}

\begin{proof}The algebra $\End_{G_{d(\sso\otimes\sigma )}}(\CP_{\so ,\sigma })$ is generated by operators $T_w$ and $J_r$ and an algebra of scalar operators $B_{\sso }$ (see \cite[5.10]{H1} and also \cite[A.6 and C.5]{H3} for the full orthogonal and unitary group). More precisely, \[\End_{G_{d(\so\otimes\sigma )}}(\CP_{\so ,\sigma })=\bigoplus _{w,r} B_{\sso}T_wJ_r,\]
where $w$ runs over the Weyl group $W(\CO )$ of a root system $\Sigma _{\so }$ associated to the orbit $\CO $, and $r$ over an $R$-group $R_{\so }$ associated to $\CO $. The root system $\Sigma _{\so }$ and the corresponding based root datum are a direct sum of irreducible root systems $\Sigma _{\so _i}$ associated to the orbits $\CO _i$, so that $W(\CO )$ is the product of the Weyl groups $W(\CO _i)$. The $R$-group $R_{\so }$ is the product of the $R$-groups $R_{\so _i}$ with $R_{\so _i}$ non trivial if and only if $\Sigma _{\so _i}$ is of type $D$ and then $R_{\so _i}$ is of order $2$ generated by the outer automorphism of $\Sigma _{\so _i}$. The group algebra $B_{\so }$ is the group algebra of the lattice of the root datum associated to $\CO $. Consequently, $B_{\so }=\bigotimes _iB_{\so _i}$. The operators $T_{w_i}$, $J_{r_i}$ and $b_i$ with $w_i\in W(\CO _i)$, $r_i\in R_{\so _i}$ and $b_i\in B_{\so _i}$ commute with operators associated to any $\CO _{i'}$ with $i\ne i'$ and generate a subalgebra isomorphic to $\End_{G_{d(\sso _i,\sigma )}}(\CP _{\so _i,\sigma })$ by the above embedding.
\end{proof}

\begin{theorem}\label{t:2.5}The category $\R (G_{d(\so\otimes\sigma )})_{\so\otimes\sigma }$ is equivalent to the category of right modules over the algebra
\begin{equation}\label{e:big-algebra}
\bigotimes_{i=1}^r \End_{G_{d(\sso _i,\sigma )}}(\CP _{\so _i,\sigma }).
\end{equation}
The equivalence of category is given by $V\mapsto Hom_G(\CP _{\so ,\sigma },V)$, where the algebra $\bigotimes_{i=1}^r \End_{G_{d(\sso _i,\sigma )}}(\CP _{\so _i,\sigma })$ acts via the isomorphism with $\End_{G_{d(\sso\otimes\sigma )}}(\CP_{\so ,\sigma })$.

It is compatible with parabolic induction, the Jacquet functor and preserves temperedness and discrete series.
\end{theorem}

\begin{remark}\label{r:2.6} It follows from \cite[Proposition 6.1]{H1} and \cite[A.2, C.5]{H3}, that the tensor product factors of the  algebra (\ref{e:big-algebra}) are either affine Hecke algebras of types $A$ or $B$  or the semidirect product of an affine Hecke algebra of type $D$ by the group algebra of a cyclic group of order 2.
\end{remark}

\begin{proof}The first part follows from Proposition \ref{p:2.3} by Morita equivalence, as $\CP _{\so ,\sigma }$ is a projective generator in $\R (G_{d(\so\otimes\sigma )})_{\so\otimes\sigma }$.  The compatibility with parabolic induction and Jacquet functor is a consequence of \cite[section 5]{Ro} (see also  \cite[7.9]{H1}), while preservation of temperedness and discrete series was considered in \cite{H2} (properly, \cite{H2} did not treat unitary groups, but the proof in \cite{H2} generalizes to the setting of unitary groups, taking into account \cite[Appendix C]{H3}). \end{proof}

\subsection{} Applying Theorem \ref{t:2.5} to the Proposition \ref{p:2.1}, one gets:

\begin{corollary}\label{c:2.6}The category $\CC_{I;\sigma }$ is equivalent to the category of finitely generated right modules over the algebra
$$\bigoplus _{r=1}^{\infty }\bigoplus_{\{\so _1,\dots ,\so _r\}}\bigotimes_{i=1}^r \End_{G_{d(\sso _i,\sigma )}}(\CP _{\so _i,\sigma }),$$
where the $\CO _i$ denote homogeneous orbits with distinct support in $I$.

This equivalence of categories sends a summand $V$ in $(\R (G)_{\so _1\otimes\dots\otimes\so _r,\sigma })_{\fgen}$ to the right $\bigotimes _{i=1}^r\End_G(\CP _{\so _i,\sigma })$-module $Hom_G(\CP_{\so _1\otimes\dots\otimes\so _r,\sigma },V)$. It is compatible with parabolic induction, the Jacquet functor and preserves temperedness and discrete series\end{corollary}

\subsection{} We will now apply the theory of the tensor product of abelian categories as established in \cite{D} and used in \cite{H3}.

\begin{corollary}\label{c:2.7}
If $I$ and $J$ are disjoint, then one has an equivalence of category
$$\Phi _{I,\sigma ;J,\sigma }:\CC_{I;\sigma }\otimes \CC_{J;\sigma }\rightarrow \CC_{I\cup J;\sigma },$$
which is compatible with parabolic induction, the Jacquet functor, (conjugate)-dual and preserves temperedness and discrete series.
\end{corollary}

\begin{remark} (i) C. Jantzen proved in \cite[Theorem 9.3]{J} the underlying bijection of irreducible representations, satisfying a list of properties. These properties are consequences of the exactness of a functor of equivalence of categories and compatibility with parabolic induction and Jacquet functor (using among others the definition of Aubert-duality in the context of affine Hecke algebras). Of course, by exactness, one has also an equivalence between the subcategories of finite length representations.

In addition, C. Jantzen considered the more precise setting of \rm real lines \it and not whole inertial orbits, meaning that $\CO $ is replaced by the set $\CO _{\Bbb R}$ of twists by \rm real valued \it unramified characters of a given irreducible supercuspidal representation and its (conjugate-)dual. Using Lusztig's reduction theorem (see Theorem \ref{t:reduction} in this paper) and preservation of Jacquet modules, Corollary \ref{c:2.7}  and its proof remain however valid, when $I$ and $J$ are distinct sets formed by real lines, getting the precise situation considered by C. Jantzen in \cite{J}.

(ii) The preservation of unitarity is more difficult, as presently it is only a conjecture that the equivalence of category between smooth representations of $p$-adic groups and modules over affine Hecke-algebras as given by Bernstein's progenerators preserves unitarity.

(iii) The restriction to the category of finitely generated representations is due to the notion of tensor product for categories of modules over coherent rings as defined in \cite[Proposition 5.3]{D}, where this restriction is made. \end{remark}

\begin{proof}The category on the right is equivalent to the category of finitely generated right modules over the algebra
\begin{equation}\label{e:star}
\bigoplus _{t=1}^{\infty }\bigoplus_{\{\so _1,\dots ,\so _t\}}\bigotimes_{i=1}^t \End_{G_{d(\so _i\otimes\sigma )}}(\CP _{\so _i,\sigma }),
\end{equation}
where the $\CO _i$ denote self-dual inertial orbits of length $1$ with distinct support in $I\cup J$. We can write this also as
$$\bigoplus _{r=1}^{\infty }\bigoplus _{s=1}^{\infty }\bigoplus_{\{\so _1,\dots ,\so _r\}}\bigoplus_{\{\so '_1,\dots ,\so '_s\}}\bigotimes_{i=1}^r\End_{G_{d(\so _i\otimes\sigma )}}(\CP _{\so _i,\sigma })\otimes\bigotimes_{j=1}^s\End_{G_{d(\so _j'\otimes\sigma )}}(\CP _{\so _j',\sigma }),$$ where the $\CO _i$'s have support in $I$ and the $\CO _j$' have support in $J$. This is isomorphic to $$(\bigoplus _{r=1}^{\infty }\bigoplus_{\{\so _1,\dots ,\so _r\}}\bigotimes_{i=1}^r\End_{G_{d(\so _i\otimes\sigma )}}(\CP _{\so _i,\sigma }))\otimes (\bigoplus _{s=1}^{\infty }\bigoplus_{\{\so '_1,\dots ,\so '_s\}}\bigotimes_{j=1}^s\End_{G_{d(\so _j'\otimes\sigma )}}(\CP _{\so _j',\sigma })).$$
The category of finitely generated right modules over each of the factor of the above tensor product is respectively equivalent to $\CC_{I;\sigma }$ and $\CC_{J;\sigma }$. It remains to show that the category of finitely generated modules over the tensor product is equivalent to $\CC_{I;\sigma }\otimes \CC_{J;\sigma }$ (and in particular, that this tensor product exists).

In fact, each side of the tensor product can be respectively rewritten as $$\bigcup_{N=1}^{\infty}\bigoplus _{r=1}^N\bigoplus_{\{\so _1,\dots ,\so _r\}}\bigotimes_{i=1}^r\End_{G_{d(\so _i\otimes\sigma )}}(\CP _{\so _i,\sigma })$$ and $$\bigcup_{M=1}^{\infty}\bigoplus _{s=1}^M\bigoplus_{\{\so _1',\dots ,\so _s'\}}\bigotimes_{j=1}^s\End_{G_{d(\so _j'\otimes\sigma )}}(\CP _{\so _j',\sigma }),$$ where the algebras over which the union is taken are considered to be subalgebras of each others. This is a union over coherent algebras by \cite[B.2]{H3}, so that the tensor product exists \cite[B.5]{H3} and $\CC_{I;\sigma }\otimes \CC_{J;\sigma }$ is the categorical union of the corresponding finitely generated full subcategories.

The last part of the corollary is then immediate by corollary \ref{c:2.6}, except the assertion about the (conjugate-)dual. But this is obvious, as (conjugate-)dual passes to the supercuspidal support.\end{proof}

\begin{remark} Denoting $\B _{I\cup J,\sigma }$ the algebra defined in (\ref{e:star}), the proof of the  corollary is a consequence of the isomorphism $\B _{I\cup J,\sigma }\simeq\B _{I,\sigma }\otimes \B _{J,\sigma }$ explained above. For further reference, let us write $\B^{GL}_I$ if $G$ is a general linear group.
\end{remark}

\subsection{}Remark that parabolic induction defines a functor $M_{I,\sigma }:\CC_{I}^{GL}\otimes\CC_{I;\sigma }\rightarrow \CC_{I;\sigma }$.

Let $\CO _1,\dots ,\CO _r$ be homogeneous orbits in $I$ and write $\CO _i=\CO_i'\otimes\CO_i''$ with $\CO_i'$ or $\CO_i''$ possibly empty (in that case $\CO _i$ equals the non empty factor). Put $n'=\sum _il(\CO' _i)n(\CO' _i)$, $\CO '=\CO _1'\otimes\cdots\otimes\CO _r'$ and similar $n, n''$ and $\CO $, $\CO ''$. Write $d=d_0+n$ and $d''=d_0+n''$.

Then $\CO $, $\CO'$ and $\CO ''$ correspond respectively to  Bernstein components in $\R(G_d)$, $\R (GL_{n'}(E))$ and $\R (G_{d''})$. Let us denote by $\CP _{\so ,\sigma }^{G_d}$, $\CP _{\so '}^{GL_{n'}}$ and $\CP _{\so '',\sigma }^{G_{d''}}$ our projective generators in these Bernstein components.

The functor $M_{I,\sigma }$ restricted to $(\R (GL_{n'}(E)_{\so '})_{\fgen}\otimes (\R (G_{d''})_{\so ''\otimes\sigma })_{\fgen}$ sends an object $V$ to the corresponding parabolic induced representation in $(\R (G_{d})_{\so \otimes\sigma })_{\fgen}$.

Put $\A _{I,\sigma }=\B^{GL}_I\otimes\B _{I,\sigma }$. As remarked in the proof of Corollary \ref{c:2.7}, the category  $\CC_{I;\sigma }$ is equivalent to the category of finitely generated right $\B _{I,\sigma }$-modules, and by this one sees that the category $\CC_{I}^{GL}\otimes\CC_{I;\sigma }$ is equivalent to the category of finitely generated right-$\A _{I,\sigma }$-modules. Parabolic induction defines a natural map $\A _{I,\sigma }\rightarrow\B _{I,\sigma }$, which is defined on each of the summands of $\A _{I,\sigma }$ by parabolic induction.

It follows from \cite[section 5]{Ro} that the functor $M_{I,\sigma }$ commutes by equivalence of category with the one from the category of finitely generated  right $\A _{I,\sigma }$-modules to the category of finitely generated right $\B_{I,\sigma }$-modules, given by $M\mapsto M\otimes _{\sA_{I,d} }  \B_{I,d}$.

\smallskip

Recall that two functors $F, G:\CC\rightarrow\D$ are equivalent, if there exists equivalences of categories $\Phi:\CC\rightarrow \CC$ and $\Psi:\D\rightarrow\D $, so that $G=\Psi\circ F\circ\Phi $.

\begin{theorem}\label{t:2.10} Let $I$ and $J$ be disjoint sets of self-dual length $1$ inertial orbits. Identify $(\CC^{GL}_I\otimes \CC^{GL}_J)\otimes (\CC_{I;\sigma }\otimes \CC_{J;\sigma })$ with $(\CC^{GL}_I\otimes \CC_{I;\sigma })\otimes (\CC^{GL}_J\otimes \CC_{J;\sigma })$.

Then, the functors  $M_{I\cup J,\sigma }\circ (\Phi_{I,J}^{GL}\otimes \Phi_{I,\sigma ;J,\sigma})$ and $\Phi _{I,\sigma ;J,\sigma }\circ (M_{I,\sigma}\otimes M_{J,\sigma })$ are equivalent functors from $(\CC^{GL}_I\otimes \CC^{GL}_J)\otimes (\CC_{I;\sigma }\otimes \CC_{J;\sigma })=(\CC^{GL}_I\otimes \CC_{I;\sigma })\otimes (\CC^{GL}_J\otimes \CC_{J;\sigma })$ into $\CC_{I\cup J;\sigma }$.
\end{theorem}

\begin{proof} Remark first that the identification in the second sentence of the theorem is justified by \cite[5.16]{D}.

As we are only considering equivalences of functors, we can go over - by equivalence of categories -  to the category of finitely generated right modules over the corresponding endomorphism algebras. Using the compatibility remarked above of the functors  $M_{I,\sigma }$, $M_{J,\sigma }$ and $M_{I\cup J,\sigma }$ with this equivalence of categories, this breaks down to the following tensor product identity for finitely generated $\A _{I,\sigma }$-modules $M'\otimes M''$ and finitely generated $\A _{J,\sigma }$-modules $N'\otimes N''$:

\begin{align*}
&(M'\otimes M''\otimes N'\otimes N'')\otimes_{(\sA _{I,\sigma }\otimes\sA _{J,\sigma })}  (\B _{I,\sigma }\otimes\B _{J,\sigma })\\
&\simeq ((M'\otimes M'')\otimes _{\sA _{I,\sigma }}\B _{I,\sigma })\otimes ((N'\otimes N'')\otimes _{\sA _{J,\sigma }}\B _{J,\sigma }),
\end{align*}
which is immediate.
\end{proof}

\section{Reducibility of induced modules for classical affine Hecke algebras}\label{s:red}

\subsection{The Bernstein presentation} Let $(X,X^\vee,R,R^\vee,\Delta)$ be a based root datum with finite Weyl group $W$. Let $\ell:W\to \bN$ be the length function. For every $\alpha\in R$, denote by $s_\alpha\in W$ the corresponding reflection. For each root $\alpha\in R$, let $\alpha^\vee\in R^\vee$ denote the corresponding coroot. We recall the Bernstein presentation of the affine Hecke algebra following Lusztig \cite{Lu}. Let $\lambda:\Delta\to \bZ$ and $\lambda^*:\{\alpha\in \Delta\mid \alpha^\vee\in 2X^\vee\}\to\bZ$ be $W$-invariant  functions. Let $\bbq$ be an indeterminate.

Let $\CH(W,\bbq)$ denote the finite Hecke algebra, i.e., the associative, unital $\bC[\bbq]$-algebra generated by $T_w$, $w\in W$ subject to the relations:
\begin{enumerate}
\item $T_wT_{w'}=T_{ww'}$, if $\ell(ww')=\ell(w)+\ell(w')$;
\item $(T_{s_\alpha}-\bbq^{\lambda(\alpha)})(T_{s_\alpha}+1)=0$, for all $\alpha\in\Delta$.
\end{enumerate}

\begin{definition}\label{d:bernstein}
The affine Hecke algebra $\CH(\bbq)=\CH(\bbq,\lambda,\lambda^*)$  is the unique associative unital $\bC[\bbq]$-algebra generated by $T_w$, $w\in W$, and $\theta_x$, $x\in X$ subject to finite Hecke algebra relations and the relations:
\begin{enumerate}
\item $\theta_x\theta_{x'}=\theta_{x+x'}$, $x,x'\in X$;
\item $\theta_xT_{s_\alpha}-T_{s_\alpha}\theta_{s_\alpha(x)}=(\theta_x-\theta_{s_\alpha(x)})(\C G(\alpha)-1)$, where
\begin{equation}
\C G(\alpha)=\begin{cases} \frac{\theta_\alpha \bbq^{\lambda(\alpha)}-1}{\theta_\alpha-1},&\text{if }\alpha^\vee\notin 2 X^\vee,\\
 \frac{(\theta_\alpha \bbq^{(\lambda(\alpha)+\lambda^*(\alpha))/2}-1)(\theta_\alpha \bbq^{(\lambda(\alpha)-\lambda^*(\alpha))/2}+1)}{\theta_{2\alpha}-1},&\text{if }\alpha^\vee\in 2 X^\vee.
\end{cases}
\end{equation}
\end{enumerate}
\end{definition}
Denote $\CA=\bC[\bbq]\langle \theta_x:x\in X\rangle$; this is an abelian subalgebra of $\CH(\bbq)$. The center of $\CH(\bbq)$ is $\C Z=\CA^W$ \cite[Proposition 3.11]{Lu}. Let $\widehat \CA$ denote the field of fractions of the integral domain $\CA.$ Set $\widehat \CH(q)=\CH(q)\otimes_{\CA}\widehat\CA$. The algebra $\widehat\CH(q)$ contains $\CH(q)$ as an $\CA$-subalgebra. Following \cite[section 5.1]{Lu}, define the elements $\tau_\alpha\in \widehat\CH(\bbq)$:
\begin{equation}
\tau_\alpha=(T_{s_\alpha}+1)\C G(\alpha)^{-1}-1.
\end{equation}
Then $\tau_\alpha$ satisfy the braid relations and $\tau_\alpha^2=1$ \cite[Proposition 5.2]{Lu}. This means that the assignment $\tau_\alpha\to s_\alpha$ extend to a group isomorphism between $\<\tau_\alpha:\alpha\in\Delta\>$ and  $W$. Let $\tau_w$ denote the element corresponding to $w\in W$. Then
\begin{equation}
f\tau_w=\tau_w w^{-1}(f), \text{ for all }f\in\widehat\CA.
\end{equation}
We define elements $R_\alpha\in \CH(q)$ by clearing the denominator of $\tau_\alpha$:
\begin{equation}
R_\alpha=\begin{cases} T_{s_\alpha}(\theta_\alpha-1)-(\bbq^{\lambda(\alpha)}-1)\theta_\alpha,&\text{if }\alpha^\vee\notin 2X^\vee,\\
(T_{s_\alpha}+1)(\theta_{2\alpha}-1)-(\theta_\alpha \bbq^{(\lambda(\alpha)+\lambda^*(\alpha))/2}-1)(\theta_\alpha \bbq^{(\lambda(\alpha)-\lambda^*(\alpha))/2}+1),&\text{if }\alpha^\vee\in 2X^\vee.
\end{cases}
\end{equation}
The elements $R_\alpha$ satisfy the following similar properties to $\tau_\alpha$, as can be seen from the commutativity of $\CA$ (and the properties of $\tau_\alpha$).

\begin{lemma}\label{l:R-alpha}
\begin{enumerate}
\item $R_\alpha^2= (\bbq^{\lambda(\alpha)}-\theta_{-\alpha})(\bbq^{\lambda(\alpha)}-\theta_\alpha)$, if $\alpha^\vee\notin 2X^\vee,$ and
$R_\alpha^2=(\bbq^{(\lambda(\alpha)+\lambda^*(\alpha))/2}-\theta_\alpha)(\bbq^{(\lambda(\alpha)-\lambda^*(\alpha))/2}+\theta_\alpha)(\bbq^{(\lambda(\alpha)+\lambda^*(\alpha))/2}-\theta_{-\alpha})(\bbq^{(\lambda(\alpha)-\lambda^*(\alpha))/2}+\theta_{-\alpha})$ if $\alpha^\vee\in 2X^\vee$.
\item $f R_\alpha=R_\alpha s_{\alpha}(f)$, for all $f\in \CA$.
\item $\{R_\alpha\}$ satisfy the braid relations, and hence we may define $R_w$, $w\in W$.
\end{enumerate}
\end{lemma}
Let $V$ be a finite dimensional module of $\CH(\bbq)$ on which $\bbq$ acts by $q\in \bC^\times$. (We will be interested in the case $q>1$.) A weight vector of $V$ with weight $\mu\in X^\vee\otimes_\bZ \bC^\times$ is a vector $0\neq v\in V$ such that
\begin{equation}
\theta_x\cdot v=\<x,\mu\>v,\text{ for all } x\in X.
\end{equation}
Notice that by Lemma \ref{l:R-alpha}, if $v$ is such a weight vector and if $R_w\cdot v\neq 0$ then $R_w\cdot v$ is  also a weight vector with weight $w(\mu)$.

\subsection{Rank one} Consider the $SO(3,\bC)$ root datum with two parameters. More precisely, $X=\bZ\alpha$, $X^\vee=\bZ \alpha^\vee/2$, $R=\{\pm\alpha\}$, $R^\vee=\{\pm\alpha^\vee\}$, $\Delta=\{\alpha\}$, and set $\lambda=\lambda(\alpha)$ and $\lambda^*=\lambda^*(\alpha)$. Specialize $\bbq=q>1$. The affine Hecke algebra $\CH_{A_1}(\lambda,\lambda^*)$ is generated by $T=T_{s_\alpha}$ and $\theta=\theta_{\alpha}$, $\theta^{-1}$ subject to the relations
\begin{equation}
\begin{aligned}
&(T+1)(T-q^\lambda)=0,\\
&\theta T-T\theta^{-1}=(q^\lambda-1)\theta+(q^{(\lambda+\lambda^*)/2}-q^{(\lambda-\lambda^*)/2}).
\end{aligned}
\end{equation}
Let $\CA=\bC[\theta,\theta^{-1}]$. The only proper parabolically induced modules are the two-dimensional modules:
\begin{equation}
I(\nu)=\CH_{A_1}(\lambda,\lambda^*)\otimes_\CA \bC_\nu,\quad \nu\in\bC/(2\pi i\bZ/\log q),
\end{equation}
where $\bC_\nu$ is the one-dimensional module of $\CA$ on which $\theta$ acts by the scalar $q^\nu$. There are four one-dimensional $\CH_{A_1}(\lambda,\lambda^*)$-modules:
\begin{equation}
\begin{aligned}
& [T=-1, \theta=q^{-(\lambda+\lambda^*)/2}],& [T=-1, \theta=-q^{-(\lambda-\lambda^*)/2}],\\
& [T=q^\lambda, \theta=q^{(\lambda+\lambda^*)/2}],& [T=q^\lambda, \theta=-q^{(\lambda-\lambda^*)/2}].
\end{aligned}
\end{equation}
The first of these modules is the analogue of the Steinberg representation and the third is the analogue of the trivial representation.
The following result is well known and easy to prove (for example by using the description of the one-dimensional modules above).

\begin{lemma}\label{l:rank-one1}
The module $I(\nu)$ is reducible if and only if $q^\nu\in\{q^{\pm(\lambda+\lambda^*)/2}, -q^{\pm (\lambda-\lambda^*)/2}\}$.
\end{lemma}

Now consider the $SL(2,\bC)$ root datum with one parameter. In this case, $X=\bZ\alpha/2$, $X^\vee=\bZ \alpha^\vee$, $R=\{\pm\alpha\}$, $R^\vee=\{\pm\alpha^\vee\}$, $\Delta=\{\alpha\}$, and set $\lambda=\lambda(\alpha)$. Specialize $\bbq=q>1$. The affine Hecke algebra $\CH_{A_1}(\lambda)$ is generated by $T=T_{s_\alpha}$ and $\theta=\theta_{\alpha/2}$, $\theta^{-1}$ subject to the relations
\begin{equation}
\begin{aligned}
&(T+1)(T-q^\lambda)=0,\\
&\theta T-T\theta^{-1}=(q^\lambda-1)\theta.
\end{aligned}
\end{equation}
Let $\CA=\bC[\theta,\theta^{-1}]$. As before, the only proper parabolically induced modules are the two-dimensional modules:
\begin{equation}
I(\nu)=\CH_{A_1}(\lambda)\otimes_\CA \bC_\nu,\quad \nu\in\bC/(2\pi i\bZ/\log q).
\end{equation}
There are four one-dimensional $\CH_{A_1}(\lambda)$-modules:
\begin{equation}
[T=-1, \theta=q^{-\lambda/2}],\  [T=-1, \theta=-q^{-\lambda/2}],\  [T=q^\lambda, \theta=q^{\lambda/2}],\ [T=q^\lambda, \theta=-q^{\lambda/2}].
\end{equation}
Notice that in the first two modules, $\theta_\alpha=\theta^2=q^{-\lambda}$ and in the last two $\theta_\alpha=q^{\lambda}$.
Again the following result is well known.

\begin{lemma}\label{l:rank-one2}
The module $I(\nu)$ is reducible if and only if $q^\nu\in\{\pm q^{\pm \lambda/2}\}$.
\end{lemma}

\subsection{Type $B_n$} (Hecke algebra with root datum of type $SO(2n+1)$ and unequal parameters) Let us specialize now to the case when $X=\bZ^n$ and $R$ is the root system of type $B_n$. Let $\Delta=\{\epsilon_1-\epsilon_2,\epsilon_2-\epsilon_3,\dots,\epsilon_{n-1}-\epsilon_n,\epsilon_n\}$ in the usual coordinates. Let $m_+,m_-$ be real parameters and set $\bbq=q>1$,
\begin{equation}
\lambda(\epsilon_i-\epsilon_{i+1})=q,\ \lambda(\epsilon_n)=q^{m_++m_-},\ \lambda^*(\epsilon_n)=q^{m_+-m_-}.
\end{equation}
Denote by $\CH_B=\CH_B(m_+,m_-)$ this specialization of the affine Hecke algebra. Then $\CH_B$ is generated by $T_i=T_{\alpha_i}$ and $\theta_i=\theta_{\epsilon_i}$ subject to the (finite) Hecke relations and the cross-relations:
\begin{equation}
\begin{aligned}
&\theta_i T_i-T_i\theta_{i+1}=(q-1)\theta_i, &1\le i\le n-1,\\
&\theta_i T_j=T_j \theta_i, &|i-j|\ge 2,\\
&\theta_{n-1} T_n=T_n \theta_{n-1},\\
&\theta_n T_n-T_n\theta_n^{-1}=(q^{m_++m_-}-1)\theta_n+(q^{m_+}-q^{m_-}).
\end{aligned}
\end{equation}
Then $\CA=\bC[\theta_i^{\pm 1}, 1\le i\le n]$. The elements $R_\alpha$ become
\begin{equation}\label{e:R_n}
\begin{aligned}
&R_{i}=T_i(\theta_{\alpha_i}-1)-(q-1)\theta_{\alpha_i},\text{ where }\theta_{\alpha_i}=\theta_i \theta_{i+1}^{-1}, \ 1\le i\le n-1,\\
&R_n=(T_n+1)(\theta_n^2-1)-(\theta_n q^{m_+}-1)(\theta_n q^{m_-}+1).
\end{aligned}
\end{equation}
We remark that with our notation
\begin{equation}\label{e:R_n-squared}
R_n^2=(q^{m_+}-\theta_n)(q^{m_+}-\theta_n^{-1})(q^{m_-}+\theta_n)(q^{m_-}+\theta_n^{-1}).
\end{equation}

\subsection{Induced modules: overlapping supports}
Let $\CH_A$ denote the subalgebra of $\CH_B$ generated by $T_i$, $1\le i\le n-1$ and by $\theta_j$, $1\le j\le n.$ This is isomorphic to the affine Hecke algebra of $\mathfrak{gl}(n)$. We will think of the weights of $\CH_A$-modules and of $\CH_B$-modules as $n$-tuples
\[\nu=(q^{\lambda_1},q^{\lambda_2},\dots,q^{\lambda_n}),
\]
where $\lambda_i\in \bC/{2\pi i\log(q)}$, $1\le i\le n$. The Weyl group $W_n$ of type $B_n$ acts by permutations and inversions on such a $\nu$.

Let $W_n^{S_n}$ denote the set of (unique) minimal length representatives of the cosets $W_n/S_n$. Let $\pi$ be a finite dimensional $\CH_A$-module and form the parabolically induced $\CH_B$ module
\[I(\pi)=\CH_B\otimes_{\CH_A}\pi.
\]
The following lemma is a particular case of a well-known result, see for example \cite{BM}.

\begin{lemma}\label{l:BM}
The weights of $I(\pi)$ are $w(\nu)$, where $\nu$ ranges over the weights on $\pi$ and $w\in W_n^{S_n}$.
\end{lemma}

Notice that in particular $\{s_n,s_{n-1}s_n,s_{n-2}s_{n-1}s_n,\dots,s_1 s_2\dotsb s_n\}\subset W_n^{S_n}$. This means that if $\nu=(q^{\lambda_1},q^{\lambda_2},\dots,q^{\lambda_n})$ is a weight of $\pi$, then every $(q^{\lambda_1},\dots,q^{\lambda_{i-1}},q^{-\lambda_n},q^{\lambda_{i+1}},\dots,q^{\lambda_{n-1}}, q^{\lambda_i})$ is a weight of $I(\pi)$.

Before proving our reducibility result, we need a lemma about finite Hecke algebras. Let $\CH_{W_n}$ be the finite Hecke algebra for $W_n$, and $\CH_{S_n}\subset \CH_{W_n}$ the finite Hecke algebra for $S_n$.
Denote by $\CH_{W_1}\subset \CH_{W_n}$ the subalgebra generated by $T_n$.

Notice that the map \[T_i\mapsto T_i,\ 1\le i\le n-1, \quad T_n\mapsto q\] defines a surjective algebra homomorphism $\CH_{W_n}\twoheadrightarrow \CH_{S_n}$. If $\sigma$ is a simple $\CH_{S_n}$-module, let $\sigma\times 0$ denote the simple $\CH_{W_n}$-module obtained by the pull-back.  Similarly, there is another surjective algebra homomorphism $\CH_{W_n}\twoheadrightarrow \CH_{S_n}$ defined by $T_i\mapsto T_i,\ 1\le i\le n-1, \quad T_n\mapsto -1$. Denote by $0\times\sigma$ the resulting $\CH_{W_n}$-module.

\begin{lemma}\label{l:finite} Let $\sigma$ be a simple $\CH_{S_n}$-module and $\sigma_0$ a simple (trivial or Steinberg) $\CH_{W_1}$-module. Suppose \[\phi: \CH_{W_n}\otimes_{\CH_{W_1}}\sigma_0\to \CH_{W_n}\otimes_{\CH_{S_n}}\sigma\] is an $\CH_{W_n}$-homomorphism. Then $\phi$ is not surjective.
\end{lemma}

\begin{proof} We use the fact that both $\sigma\times 0$ and $0\times\sigma$ occur in $\CH_{W_n}\otimes_{\CH_{S_n}}\sigma$.
Suppose first that $\sigma_0$ is the trivial $\CH_{W_1}$-module. Then every simple $\CH_{W_n}$-module that appears in $\CH_{W_n}\otimes_{\CH_{W_1}}\sigma_0$ has a vector on which $T_n$ acts by $-1$. This means that $\sigma\times 0$ cannot appear in $\CH_{W_n}\otimes_{\CH_{W_1}}\sigma_0$.

On the other hand, if $\sigma_0$ is the Steinberg module, the same argument applies to  see that $0\times\sigma$ (on which $T_n$ acts by $q\Id$) does not appear in $\CH_{W_n}\otimes_{\CH_{W_1}}\sigma_0$.
\end{proof}

\begin{theorem}\label{t:reducible} Assume that $m_+\neq 0$ and $m_-\neq 0$.
Let $\pi$ be a simple $\CH_A$-module and suppose that $\pi$ has a weight $(q^{\lambda_1},\dots,q^{\lambda_n})$ such that $q^{\lambda_i}\in\{q^{\pm m_+},-q^{\pm m_-}\}$ for some $i$. Then $I(\pi)$ is reducible.
\end{theorem}

\begin{proof}
Let $j$ be the largest index such that $q^{\lambda_j}\in\{ q^{\pm m_+}, -q^{\pm m_-}\}$. Using the remark after Lemma \ref{l:BM}, there exists a weight $\nu=(q^{\lambda_1'},\dots,q^{\lambda_n'})$ of $I(\pi)$ such that $q^{\lambda_n'}\in\{ q^{\pm m_+}, -q^{\pm m_-}\}$. Let $0\neq x\in I(\pi)$ be a weight vector for this weight. Then by (\ref{e:R_n-squared})
\begin{equation}
R_n^2\cdot x=0.
\end{equation}
There are two cases.

(a) $y=R_n\cdot x\neq 0$. Then $R_n\cdot y=0$. Notice that $y$ is a weight vector with weight $s_n(\nu)$, hence $\theta_n\cdot y=q^{-\lambda_n'}y$. Using formula (\ref{e:R_n}), we get
\[(q^{-2\lambda'_n}-1) T_n\cdot y=(q^{m_++m_--2\lambda'_n}+q^{m_+-\lambda'_n}-q^{m_- -\lambda_n'}-q^{-2\lambda_n'}) y.
\]
The assumption on $m_+$ and $m_-$ implies that $(q^{-2\lambda'_n}-1)\neq 0$, hence $T_n$ acts by a scalar on $y$. Since $y$ is also a weight vector, it follows that $\bC\<y\>$ is a one-dimensional module $\pi_0$ of the parabolic Hecke subalgebra $\CH_{B_1}$ of type $B_1$.

(b) $R_n\cdot x=0$. Then set for uniformity $y=x$ and the same argument as in (a) applies with the only change being that in the above formula $\lambda_n'$ should be replaced by $-\lambda_n'$.

In conclusion, we have constructed a one-dimensional $\CH_{B_1}$-module $\pi_0=\bC y$, generated by the element $y$, inside the $\CH_B$-module $I(\pi).$ We emphasize that the parabolic subalgebra $\CH_{B_1}$ is generated by $T_{s_n}$ and the full abelian subalgebra $\CA$.  Let $U_y=\CH_B y$ denote the $\CH_B$-submodule of $I(\pi)$ generated by $y$. There is a natural $\CH_B$-homomorphism
\begin{equation}
\phi: \CH_B\otimes_{\CH_{B_1}}\pi_0\twoheadrightarrow U_y, \text{ defined by } h\otimes y\mapsto h\cdot y.
\end{equation}
By definition, this is surjective.

We claim that $U_y$ is a proper submodule of $I(\pi)=\CH_B\otimes_{\CH_A}\pi$. Assume by contradiction that $U_y=I(\pi)$. Then $\phi$ is a surjective $\CH_B$-homomorphism
\[
\phi: \CH_B\otimes_{\CH_{B_1}}\pi_0\twoheadrightarrow I(\pi).
\]
Denote by $\sigma _0$ the $\CH_{B_1}$-module generated by $y$ (the space is the same than the one for $\pi _0$) and by $\sigma $ the $\CH_{S_n}$-module generated by $y$ in the space of $\pi $. These are simple modules. The $\CH_{W_n}$-homomorphism $\CH_{W_n}\otimes_{\CH_{W_1}}\sigma_0\rightarrow \CH_{W_n}\otimes_{\CH_{S_n}} \sigma $ induced by $\phi $ is clearly surjective. But this is a contradiction with Lemma \ref{l:finite}.
\end{proof}

\begin{remark}\label{r:red}
The same reducibility result in Theorem \ref{t:reducible} and the same proof hold when $m_-=0$, but $m_+\neq 0$ and $q^{\lambda_i}\in\{q^{\pm m_+}\}$ for some $i$.
\end{remark}
The remaining case is discussed next.
\begin{theorem}\label{t:reducible2} Assume that  $m_+=m_-=0$.
Let $\pi$ be a simple $\CH_A$-module and suppose that $\pi$ has a weight $\nu=(q^{\lambda_1},\dots,q^{\lambda_n})$ such that $q^{\lambda_i}\in\{\pm 1\}$ for some $i$. Then $I(\pi)$ is reducible.
\end{theorem}

\begin{proof}
Without loss of generality, as in the proof of Theorem \ref{t:reducible}, we may assume that $i=n$.
Suppose that $m_+=m_-=0$. Then the last relation in the Bernstein presentation of $\CH_B$ becomes simply
\begin{equation}
\theta_n T_n=T_n \theta_n^{-1}.
\end{equation}
We remark that also in the finite Hecke algebra, the relation for $T_n$ is just $T_n^2=1$.

Let $x\neq 0$ be a weight vector in $I(\pi)$ with weight $\nu$. Because in our situation, $\theta_z T_n=T_n \theta_{s_n(z)}$ for all $z\in X$, we see that $T_n\cdot x$ is also a weight vector with weight $s_n(\nu)=\nu$. Since $T_n$ is invertible, $T_n\cdot x\neq 0$. If $T_n\cdot x$ is a multiple of $x$ (necessarily $\pm x$), then we set $y=x$ and we are in the same situation as in the last paragraph of the proof of Theorem \ref{t:reducible}, hence Lemma \ref{l:finite} may be used to finish the argument.

Otherwise, $x$ and $T_n\cdot x$ are linearly independent. Set
\[y=T_n\cdot x - x,\quad  \text{(also, $y=T_n\cdot x+x$ would work equally well)},
\]
which is a nonzero weight vector with weight $\nu$ and which transforms like the Steinberg (sign) representation of $\CH_{W_1}(m_+=0)\cong \bC[\bZ/2]$. Now we can finish the proof in the same way as in the proof of Theorem \ref{t:reducible}.

\end{proof}

\subsection{Induced modules: disjoint supports}\label{s:3.5} We return to the general case of the affine Hecke algebra $\CH_B=\CH_B(m_+,m_-)$. We now consider the opposite situation for an induced $\CH_B$-module $I(\pi)$, where $\pi$ is a simple $\CH_A$-module. Namely, we assume that the central character of $\pi$ (and hence of $I(\pi)$) is a Weyl group orbit of $(q^{\lambda_1},\dots,q^{\lambda_n})$, where $q^{\lambda_i}\notin \{q^{\pm m_+}, -q^{\pm m_-}\}$ for any $i$. We say that an $\CH_A$-module (or an $\CH_B$-module) has {\it separated support} if its central character satisfies this condition.

Consider the nonzero element of $\CA$:
\begin{equation}
\Omega=\prod_{i=1}^n (q^{m_+}-\theta_i)(q^{m_-}+\theta_i)(q^{m_+}-\theta_i^{-1})(q^{m_-}+\theta_i^{-1}).
\end{equation}
Since $\Omega$ is $W_n$-invariant, we have $\Omega\in\C Z$ and $\Omega\in Z(\CH_A)=\CA^{S_n}$ as well. Define
\begin{equation}
\wti{\C Z}=\C Z[\Omega^{-1}],\  \wti{Z(\CH_A)}=\CA^{S_n}[\Omega^{-1}]\quad\wti \CH_A=\CH_A\otimes_{Z(\CH_A)} \wti{Z(\CH_A)},\text{ and } \wti \CH_B=\CH_B\otimes_{\C Z} \wti{\C Z}.
\end{equation}
It is clear that there is a natural equivalence of categories between $\CH_A$-modules (respectively, $\CH_B$-modules) with separated support and $\wti\CH_A$-modules (respectively, $\wti\CH_B$-modules). We also note that $\wti H_B$ is a subalgebra of $\widehat \CH_B$.

The element $\tau_n=R_n(\theta_n q^{m_+}-1)^{-1}(\theta_n q^{m_-}+1)^{-1}$ defined previously lives in $\wti\CH_B$. Since $\tau_n$ involves in its definition only the elements $T_n$ and $\theta_n$, it is immediate that $\tau_n$ commutes with all $T_i$, $1\le i\le n-2$.

\begin{lemma}\label{l:commute-T}
In $\wti H_B$, $\tau_n T_{n-1} \tau_n T_{n-1}=T_{n-1} \tau_n T_{n-1} \tau_n$.
\end{lemma}

\begin{proof}
By Lemma \ref{l:R-alpha}, $R_{n-1}R_nR_{n-1}R_n= R_n R_{n-1} R_n R_{n-1}$, which also implies that
\[R_{n-1}\tau_n R_{n-1} \tau_n=\tau_n R_{n-1} \tau_n R_{n-1}.
\]
Here $R_{n-1}=T_{n-1}(\theta_{\alpha_{n-1}}-1)-(q-1)\theta_{\alpha_{n-1}}$. Denote $\bar\alpha_{n-1}=s_n(\alpha_{n-1})$. Notice that $s_{n-1}(\bar\alpha_{n-1})=\bar\alpha_{n-1}$, hence $T_{n-1}$ commutes with $\theta_{\bar\alpha_{n-1}}.$ We use this and the relation $\theta_x\tau_n=\tau_n\theta_{s_n(x)}$ to compute the left hand side of the formula directly:
\begin{equation}
\begin{aligned}
R_{n-1}\tau_nR_{n-1}\tau_n&=T_{n-1}\tau_n T_{n-1} \tau_n (\theta_{\alpha_{n-1}}-1)(\theta_{\bar\alpha_{n-1}}-1)-(q-1)\zeta,
\text{ where }\\
\zeta&=\tau_n T_{n-1} \tau_n \theta_{\alpha_{n-1}}(\theta_{\bar\alpha_{n-1}}-1)+ T_{n-1} (\theta_{\alpha_{n-1}}-1) \theta_{\bar\alpha_{n-1}}-(q-1) \theta_{\alpha_{n-1}}\theta_{\bar\alpha_{n-1}}.
\end{aligned}
\end{equation}
Since $\tau_n^2=1$ to obtain the right hand side, we can conjugate the above formula by $\tau_n$. Remark that $\tau_n\zeta\tau_n=\zeta$, which then leads to the identity:
\begin{equation}
T_{n-1}\tau_n T_{n-1} \tau_n (\theta_{\alpha_{n-1}}-1)(\theta_{\bar\alpha_{n-1}}-1)=\tau_nT_{n-1}\tau_n T_{n-1} (\theta_{\alpha_{n-1}}-1)(\theta_{\bar\alpha_{n-1}}-1).
\end{equation}
In $\widehat \CH_B$, we may divide both sides by $(\theta_{\alpha_{n-1}}-1)(\theta_{\bar\alpha_{n-1}}-1)$, which means that $T_{n-1}\tau_n T_{n-1} \tau_n =\tau_nT_{n-1}\tau_n T_{n-1}$ holds in $\widehat\CH_B$, but then also in $\wti\CH_B$ since the two sides are in $\wti\CH_B$ (which is a subalgebra of $\widehat\CH_B$).
\end{proof}

Define $\wti\CH_n^0$ to be the $\CA$-subalgebra of $\wti\CH_B$ generated by $T_1,\dots,T_{n-1}$, $\tau_n$.

\begin{remark}\label{r:diff}
The algebra $\wti\CH_n^0$ is not equal to $\wti\CH_B$. The reason is that in order to write $T_n$ in terms of $\tau_n$, one needs to invert $\theta_n^2-1$.
\end{remark}

Let $\CA[(\theta^2-1)^{-1}]$ denote the extension of $\CA$ where we adjoin $(\theta_i^2-1)^{-1}$ for all $1\le i\le n.$ Define $\wti\CH_n^0[(\theta^2-1)^{-1}]=\wti\CH_n^0\otimes_\CA \CA[(\theta^2-1)^{-1}]$, $\wti\CH_B[(\theta^2-1)^{-1}]=\wti\CH_B\otimes_\CA \CA[(\theta^2-1)^{-1}]$ etc.

\begin{proposition}\label{p:isom}The algebras $\wti\CH_n^0[(\theta^2-1)^{-1}]$ and $\wti\CH_B[(\theta^2-1)^{-1}]$ are isomorphic to each other, and they are naturally isomorphic to $\CH_B(0,0)[(\theta^2-1)^{-1}]$, where $\CH_B(0,0)$ is the affine Hecke algebra of type $B_n$ with parameters $m_+=m_-=0$.
\end{proposition}

\begin{proof}
The first claim follows from Remark \ref{r:diff}. For the second, the isomorphism between $\wti\CH_n^0[(\theta^2-1)^{-1}]$ and $\CH_B(0,0)[(\theta^2-1)^{-1}]$ is given by $T_i\mapsto T_i$, $1\le i\le n-1$, $\tau_n\mapsto T_n$, and $\theta_j\mapsto \theta_j$, $1\le j\le n$. Using Lemma \ref{l:commute-T}, where now we can cancel $(\theta_{\alpha_{n-1}}-1)(\theta_{\bar\alpha_{n-1}}-1)$, we see that the relations between generators are all respected.
\end{proof}

Moreover, the affine Hecke algebra  $\CH_A$  of $\mathfrak{gl}(n)$  is a subalgebra of $\wti\CH_n^0$. To summarize, we have the inclusions:
\begin{equation}
\CH_A\subset \wti\CH_n^0\subset \wti\CH_B.
\end{equation}

\begin{theorem}\label{t:reducible3}
 Let $\CH_B\supset\CH_A$ be the affine Hecke algebra of type $B_n$ with parameters $m_+$ and $m_-$. Let $\pi$ be a simple $\CH_A$-module with separated support. If $\pi$ has a central character $(q^{\lambda_1},\dots,q^{\lambda_n})$ such that $q^{\lambda_i}\notin\{\pm 1\}$ for all $i$, then $I(\pi)$ is reducible if and only if $\wti\CH_n^0\otimes_{\CH_A}\pi$ (or equivalently, $\CH_B(0,0)\otimes _{\CH_A}\pi$) is reducible.
\end{theorem}

\begin{proof} We have that $\pi$ does not have $\pm 1$ in its central character. We may invert therefore all $\theta_i^2-1$.  The module $\pi$ can be viewed as a module for $\CA[(\theta^2-1)^{-1}]$ and we have:
\begin{equation}
I(\pi)=\CH_B[(\theta^2-1)^{-1}]\otimes_{\CA[(\theta^2-1)^{-1}]}\pi.
\end{equation}
But now $\CH_B[(\theta^2-1)^{-1}]$ is naturally isomorphic to $\wti\CH_n^0[(\theta^2-1)^{-1}]$ by Proposition \ref{p:isom}. Hence
\begin{equation}
I(\pi)=\wti\CH_n^0[(\theta^2-1)^{-1}]\otimes_{\CA[(\theta^2-1)^{-1}]}\pi=\wti\CH_n^0\otimes_{\CH_A}\pi.
\end{equation}
This proves the claim.
\end{proof}

Theorem \ref{t:reducible3} says, in particular, that if $\pi$ has separated support not containing $\pm 1$, then the reducibility of $I(\pi)$ only depends on $\pi$ and not on the labels $m_+,m_-$ of the Hecke algebra of $B_n$. This proves in the regular case of separated support not containing $\pm 1$, the following theorem, which is a precise Hecke algebras formulation of \cite[Conjecture 1.3]{LT}.

\begin{theorem}\label{t:lapid}
Consider the Hecke algebras $\CH_B(m_1^+,m_1^-)$ and $\CH_B(m_2^+,m_2^-)$ and let $\pi$ be an irreducible $\CH_A$-module with separated support with respect to both algebras. Let $I(\pi)_1=\CH_B(m_1^+,m_1^-)\otimes_{\CH_A}\pi$ and $I(\pi)_1=\CH_B(m_2^+,m_2^-)\otimes_{\CH_A}\pi$ be the induced modules. Then $I(\pi)_1$ is reducible if and only if $I(\pi)_2$ is.
\end{theorem}

We prove Theorem \ref{t:lapid} in general in the next section (see Corollaries \ref{c:conj-true} and \ref{c:true2}) using Kato's exotic geometry \cite{Ka}.

\begin{example}
It is possible for $I(\pi)$ to be reducible even if $\pi$ has separated support. For example, take $\pi$ to equal an irreducible minimal principal series of $\CH_A$ with real central character $(q^{\lambda_1},q^{\lambda_2},\dots,q^{\lambda_n})$, $\lambda_i\in\bR$ for all $i$, such that:
\begin{equation}
\lambda_i-\lambda_j\neq 1,\text{ for all }i\neq j,\ \lambda_i\neq \pm m_+,\text{ for all } i.
\end{equation}
The first condition means that $\pi$ is irreducible, while the second that it has separated support. By induction in stages $I(\pi)$ is the minimal principal series of $\CH_B$ which, given the conditions already imposed, is reducible if and only if
\begin{equation}
\lambda_i+\lambda_j= \pm 1,\text{ for some }i\neq j.
\end{equation}
Notice that this reducibility condition is independent of $m_+,m_-$, as predicted by Theorem \ref{t:reducible3}.
\end{example}

\section{Kato's exotic geometry} For the comparison between the representation theories of affine Hecke algebras for the same classical root datum, but with different parameters, Kato's geometric realization \cite{Ka}, via his ``exotic geometry", is particularly well-suited. Theorem \ref{t:lapid} will follow from this geometric realization, but we need to summarize the main relevant results of Kato's work and the particular applications of it from \cite{CK}. The geometric construction in Kato's work follows the main techniques of Kazhdan-Lusztig and Ginzburg, but the main geometric object is the exotic Steinberg variety (\ref{e:Steinberg}) rather than the classical Steinberg variety. We mention that unlike the classical Kazhdan-Lusztig picture, the exotic geometry seems to not fit so well with the Langlands classification in terms of parabolical induction from tempered modules (a classification of the tempered modules in the exotic geometry is achieved in \cite{CK}, but the description is quite involved). However, as we will see, the geometric geometry is very useful when one needs to compare the representation theory of two affine Hecke algebras of the same type, but with different parameters.

To simplify notation, write $\CH_{B_n}$  for the generic affine Hecke algebra defined in terms of the root datum of $SO(2n+1,\bC)$ and with three parameters $(\bbq_0,\bbq_1,\bbq_2)$. We may drop the subscript when the context is clear. The notation of \cite{Ka} differs a little bit from our previous notation, and we will explain in subsection \ref{s:special} the relation between his parameters and the ones we used before. It suffices to say here that the affine Hecke algebra $\CH_B(m_+,m_-)$ that we considered before is obtained in the exotic geometry by specializing $(\bbq_0,\bbq_1,\bbq_2)=(-q^{m_-},q^{m_+},q).$ Notice the minus sign in front of the first entry.

\

We follow \cite{CK}. We fix some notation first. Let $G=Sp(2n,\bC)$ acting on its vector representation $V_1=\bC^{2n}$. Let $V_2={\bigwedge}^2 V_1$ and set
\[\mathbb V=V_1\oplus V_1\oplus V_2.
\]
Let $P_0$ be a Borel subgroup and $T\subset P_0$ be a maximal torus. We will write the roots of $G$ with respect to $T$ in the standard coordinates $R'=\{\pm\epsilon_i\pm\epsilon_j,\ i\neq j,\ 2\epsilon_i\}$, and we make the convention that the simple roots with respect to $P_0$ are $\Delta'=\{\epsilon_i-\epsilon_{i+1},\ 1\le i\le n-1,\ 2\epsilon_n\}$.

 Let $\C G=G\otimes (\bC^\times)^3$ which acts on $\mathbb V$ as follows: $G$ acts diagonally, and $(c_0,c_1,c_2)\in (\bC^\times)^3$ acts by multiplication by $(c_0^{-1},c_1^{-1},c_2^{-1})$. 

Let $\mathbb V^+$ denote the sum of the positive weight spaces in $\mathbb V$ and set $F=G\times_{P_0} \mathbb V^+$ with the action map
\begin{equation}
\mu: F\to \mathbb V,\quad (g,v^+)\mapsto g\cdot v^+.
\end{equation}
Let $\mathfrak N$ denote the image of $\mu$.  If $a=(s,q_0,q_1,q_2)$ is a semisimple element in $\C G$, denote by $F^a$, $\mathfrak N^a$ etc. its fixed points, and let $G(a)=G(s)$ denote the centralizer of $s$ in $G$ (this is a connected group, as $G$ is simply connected). Let
\begin{equation}
pr_{P_0}:F\to G/P_0,\quad (g,v^+)\mapsto gP_0,
\end{equation}
denote the projection onto the flag variety. The exotic Springer fiber is
\[\C E_X^a=pr_{P_0}(\mu^{-1}(X)^a)\subset G/P_0,\quad X\in \mathfrak N.
\]
Denote \[\CH_a=\CH\otimes_{Z(\CH)}\bC_a\] the specialization of the generic Hecke algebra $\CH$ at the central character defined by $a$.

Let
\begin{equation}\label{e:Steinberg}
Z=F\times_{\fk N} F\cong\{(g_1 P_0,g_2 P_0,X): g_1 P_0,g_2 P_0\in G/P_0,\ X\in g_1 \mathbb V^+\cap g_2\mathbb V^+\}
\end{equation}
be the exotic Steinberg variety \cite[section 1.1]{Ka}. The $\CG$-equivariant $K$-theory of $Z$, $K^\CG(Z)$ has a natural structure of an associative ring (\cite[section 2.7]{CG}). Using the elements of the theory developed in \cite{CG}, the first main result of \cite{Ka} is that
\begin{equation}
\CH\cong \bC\otimes_\bZ K^\CG(Z)\text{ and }\CH_a\cong \bC\otimes_\bZ K^\CG(Z^a)\text{ as }\bC\text{-associative algebras},
\end{equation}
see \cite[Theorems 2.8 and 2.11]{Ka}.

\subsection{Real positive central character}\label{s:real positive}
In this subsection, we discuss the case of real positive central character. The general case will be reduced to this one in subsection \ref{s:general}.
Every $s\in T$ admits a unique polar decomposition $s=s_e s_h$, where $s_e$ is elliptic (compact) and $s_h$ is hyperbolic.
\begin{definition}
We say that a semisimple element $a\in \C G$ is real positive if $a=(s,-1,q^m,q)$, where $q>1$, $m\in\bR$, and $s\in G$ is hyperbolic. A simple $\CH$-module $L$ is said to have real positive central character if the central character of $L$ is $W\cdot a$, where $a$ is real positive.
\end{definition}
Suppose from now on that $a\in\C G$ is a positive real semisimple element. Let $\widehat W_n$ denote the set of (isomorphism classes) of irreducible $W_n$-representations.  By the Tits deformation theorem, we know that the simple modules of the finite Hecke algebra $\CH_{W_n}$ are in one-to-one correspondence with $\widehat W_n$. By abusing notation, we may denote by $\sigma$ both the irreducible $W_n$-representation and the corresponding $\CH_{W_n}$-module. For $X\in\mathfrak N^a$, the exotic standard $\CH_a$-module is realized in the total Borel-Moore homology:
\begin{equation}
M_{(a,X)}=H_\bullet(\C E^a_X,\bC).
\end{equation}
Let $a_0=(1,-1,1,1)\in \C G$. The exotic nilpotent cone is $\mathfrak N^{a_0}$ on which $G$ acts with finitely many orbits. The parameterization of the orbits is described in \cite[section 1.3]{Ka}, and we will recall it in the next subsection. An important feature of this geometry is that for each $X\in \mathfrak N^{a_0}$,
\begin{equation}
\textup{Stab}_{G}(X)\text{ is connected (\cite[Proposition 4.5]{Ka})},
\end{equation}
where $\textup{Stab}_{G}(X)=\{g\in G\mid g\cdot X=X\}.$
This will have the effect that in this theory, unlike the classical settings of the Kazhdan-Lusztig theory \cite{KL} or Springer theory \cite{Sp}, there are only trivial local systems present in the parameterizations. The following exotic Springer correspondence should be compared with the classical Springer correspondence \cite{Sp}.

\begin{theorem}[{\cite[Theorem 8.3]{Ka}}] Let $X\in\mathfrak N^{a_0}$ be given. The top Borel-Moore homology group $H_{2 d_X}(\C E^{a_0}_X,\bC)$ has a structure of an irreducible $W_n$-representation, denoted by $\sigma_X.$
 This induces a one-to-one correspondence
\begin{equation}
G\backslash \mathfrak N^{a_0}\longleftrightarrow \widehat W_n,\quad X\mapsto \sigma_X.
\end{equation}
\end{theorem}
Since the $G$-action on $\mathfrak N^{a_0}$ has finitely many orbits, there exists a unique open dense orbit $\fk O_0^{a_0}$ in $\mathfrak N^{a_0}$. The exotic Springer correspondence is such that if $X\in \fk O_0^{a_0}$, then $\sigma_X$ is the sign representation, while if $X=0$, then $\sigma_0$ is the trivial representation.

 Let $X\in \mathfrak N^{a}$ be given. Under the assumption that $a$ is positive real, the restriction of $M_{(a,X)}$ to $\CH_{W_n}$ is given by \cite[Theorem 9.2]{Ka}, see \cite[Corollary 1.19]{CK}:
\begin{equation}
 M_{(a,X)}|_{W_n}=H_\bullet(\C E^{a_0}_X,\bC).
 \end{equation}
  By the exotic Springer correspondence, this means that $\sigma_X= H_{2 d_X}(\C E^{a_0}_X,\bC)$ appears with multiplicity one in $M_{(a,X)}$.
 \begin{definition}
For  $X\in\mathfrak N^{a_0}$, denote by $L_{(a,X)}$ the unique irreducible subquotient of $M_{(a,X)}$ containing $\sigma_X$. We will refer to $L_{(a,X)}$ as the exotic simple module. \end{definition}

\begin{remark} Since $H_0(\C E^{a_0}_X,\bC)$ is the sign $W_n$-representation, it is also interesting to remark that in this construction, every exotic standard module $M_{(a,X)}$, $X\in\mathfrak N^{a_0}$, contains the sign $\CH_{W_n}$-representation!
\end{remark}

We remark that when $a$ is a positive real semisimple element, then $\mathfrak N^a\subset\mathfrak N^{a_0}$, which allows us to use the main classification result of \cite{Ka}, in the case of positive real central character, as follows.

\begin{theorem}[{\cite[Theorem 10.1]{Ka}}]\label{t:exotic-DLL}
Let $a$ be a positive real semisimple element. There is a natural equivalence of categories \[\C F_a: \mathsf{PSh}_{G(a)}(\mathbb V^a)\longrightarrow \CH_a\text{-modules},\] where $\mathsf{PSh}_{G(a)}(\mathbb V^a)$ is the category of $G(a)$-equivariant perverse sheaves on $\mathbb V^a$.
In particular, there is a one-to-one correspondence
\[G(a)\backslash \mathfrak N^a\longleftrightarrow \text{simple }\CH_a\text{-modules},\quad X\mapsto L_{(a,X)}.
\]
\end{theorem}
An essential observation for the second part of Theorem \ref{t:exotic-DLL} is that for every $v\in \mathbb V^a$, the stabilizer $\mathsf{Stab}_{G(a)}(v)$ is a connected group \cite[Theorem 4.10]{Ka}, and therefore every $G(a)$-local system supported on the $G(a)$-orbit of $v$ is a constant sheaf.

\smallskip

Notice that the assumption that $a=(s,-1,q^m,q)$ is real positive means that $\mathbb V^a$ equals the fixed points of the action of $s\times (q^m,q)$ on $V_1\oplus V_2$. This is because the action of $-1$ on $V_1$ does not have nontrivial fixed points for hyperbolic $s$.

\begin{definition}\label{d:exotic-separated}
We will say that $a=(s,-1,q^m,q)$ is a positive real semisimple separated element if $s=(q^{\lambda_1},\dots, q^{\lambda_n})$, $\lambda_i\in \bR\setminus\{\pm m\}$, for all $1\le i\le n$.
\end{definition}

\begin{remark}\label{r:compatible}
The case of positive real semisimple characters $a=(s,-1,q^m,q)$ in the exotic geometry corresponds to the affine Hecke algebra $\CH_B(m_+,m_-)$ with $m_-=0$ and $m_+=m$ from section \ref{s:red}. Thus Definition \ref{d:exotic-separated} gives precisely the separatedness condition from section \ref{s:red} for these values of $m_-$ and $m_+$.
\end{remark}


\begin{corollary}\label{c:exotic-separated}
Let $a$ be a positive real semisimple separated element. Then $ \mathbb V^a$ equals the $q$-eigenspace $V_2(q)$ of $s$ on $V_2$, and therefore $G(a)\backslash \mathbb V^a=G(s)\backslash V_2(q)$ is independent of $m$.
\end{corollary}

\begin{proof}
Since $q^{\pm m}$ does not appear in $s$, it is immediate that $q^{-m} s$ does not have any nonzero fixed points in $V_1$. The claim follows.
\end{proof}

\begin{corollary}\label{c:equiv}
Let $m_1,m_2$ be two real numbers and let $s$ be a hyperbolic semisimple element. Set $a_i=(s,-1,q^{m_i},q)$, $i=1,2$ and suppose that both $a_1$ and $a_2$ are separated. Then $\C F=\C F_{a_1}\circ \C F_{a_2}^{-1}$ is an equivalence of categories between $\CH_{a_2}$-modules and $\CH_{a_1}$-modules.
\end{corollary}

\begin{proof}
This is clear from Corollary \ref{c:exotic-separated} and Theorem \ref{t:exotic-DLL}.
\end{proof}
Two (partially-ordered) lattices $(\C L_1,\le_1)$ and $(\C L_2,\le_2)$ are said to be isomorphic if there exists a bijection $\kappa:\C L_1\to \C L_2$ such that $\kappa(l)\le_2 \kappa(l')$ if and only if $l\le_1  l'$. We regard the set of submodules of a finite length module as a lattice with the partial order given by inclusion.

\begin{lemma}\label{l:4.10}
Let $\C C$ and $\C D$ be abelian categories of modules (over possibly different algebras) and $\C F:\C C\to \C D$ an equivalence of categories. Then $\C F$ induces an equivalence of categories between the corresponding full subcategories of modules of finite length. In addition, for every finite length module $X$ in $\C C$, $X$ and $\C F(X)$ have the same length and isomorphic lattices of submodules.
\end{lemma}

\begin{proof}
Let us first remark that an equivalence of abelian categories is always an exact functor. (It has a left and a right adjoint \cite[${\S}$1 Theorem 5.3]{P}, which implies that it is right and left exact \cite[${\S}$3 Corollary 2.3]{P}.) In particular, $\C F$ preserves the length of a module and the injectivity of a morphism. Denote the corresponding modules by $M$ and $M'=\C F(M)$ and denote by $M_1$ the direct sum of simple submodules of $M$ and by $M_1'$ the corresponding submodule of $M'$. By considering the quasi-inverse of $\C F$, one sees that the multiplicities are preserved. Taking for each simple submodule the quotient-module, reconsidering the sum of simple submodules of these quotient-modules and repeating this process, the lemma follows. (The procedure has to finish because of the finite length.)
\end{proof}

\begin{corollary}\label{c:conj-true}
Retain the same notation as in Corollary \ref{c:equiv}. Let $\pi$ be a simple $\CH_A$-module with central character $s$ and form the induced modules $I(\pi)_i=\CH_{a_i}\otimes_{(\CH_A)_{s}} \pi$, $i=1,2$. Then the modules $I(\pi)_1$ and $I(\pi)_2$ have the same length and the functor $\C F$ induces an isomorphisms of their lattices of submodules.
\end{corollary}

\begin{proof}
We have the following diagram
\begin{equation}\label{diagram}
\xymatrix{\CH_{a_2}\text{-mod} \ar[r]^{\C F} &\CH_{a_1}\text{-mod} \\
(\CH_A)_{s}\text{-mod}\ar[u]^{\Ind_A}\ar[r]_{=} &(\CH_A)_{s}\text{-mod}\ar[u]_{\Ind_A}.
}
\end{equation}
Here $\Ind$ denote the parabolic induction functor from $\CH_A$. The correspondence $\C F$ is compatible with the parabolic induction, see for example, \cite[Theorem 1.22]{CK}; in other words, diagram (\ref{diagram}) is commutative.
Since the functor $\C F$ is exact,  the claim follows.
\end{proof}

\begin{remark}One may ask if it is possible to compute explicitly the composition factors and multiplicities in an induced module $I(\pi)$. This appears to be a difficult problem. One possible approach using the exotic geometry would have two steps.

Firstly, one needs to analyze the relation between $I(\pi)$ and the exotic standard modules $M_{(a,X)}$. For this, the induction theorem in this setting is \cite[Theorem 7.4]{Ka}.

Secondly, one would need to know the composition series of exotic standard modules in terms of the exotic simple modules. These are given by intersection cohomology as in the classical setting of Kazhdan-Lusztig and Ginzburg. For every $G(a)$-orbit $\CO=G(a)\cdot X$, let $IC(\CO)$ denote the corresponding intersection cohomology $G(a)$-equivariant simple perverse sheaf on $\mathfrak N^a$. Then \cite[Theorem 11.2]{Ka} says that the multiplicity of $L_{(a,X')}$ as a composition factor of $M_{(a,X)}$ is
\begin{equation}\label{e:KL}
[M_{(a,X)}: L_{(a,X')}]=\dim H^\bullet_\CO(IC(\CO')),\text{ where }\CO'=G(a)\cdot X'.
\end{equation}
In particular, if $[M_{(a,X)}: L_{(a,X')}]\neq 0$, then $\CO\subset\overline{\CO'}.$ As one can see from \cite{Ka} (cf. \cite{CK}), when $a$ is separated, the geometry underlying this calculation is very similar to the classical one of Zelevinsky \cite{Ze}.
\end{remark}

\begin{remark}\label{AS}In general, the problem of computing explicitly the multiplicities (\ref{e:KL}) is very difficult. In the classical setting of the Kazhdan-Lusztig geometry for affine Hecke algebras \cite{KL}, an algorithm does not exist (at least to our knowledge). In the setting of the geometric graded affine Hecke algebras \cite{L1}, a complete and difficult algorithm was obtained by Lusztig in \cite{L2}. In the particular case of the graded affine Hecke algebra $\mathbb H_A$ of type $A$ (which appears for example for the $p$-adic group $GL(n)$), there exist exact functors defined by Arakawa and Suzuki \cite{AS} from the category $\C O$ for $\mathfrak{gl}(n,\bC)$ to the category of finite-dimensional $\bH_A$-modules, and, using results of Lusztig and Zelevinsky, the multiplicity question for $\bH_A$-modules can be resolved in terms of the usual Kazhdan-Lusztig polynomials for $S_n$ (i.e.,  Schubert cells for $GL(n,\bC)$). For graded affine Hecke algebras of classical types (which appear in the representation theory of classical $p$-adic groups), certain functors were defined in \cite{CT} from the category of $(\fg,K)$-modules of classical real reductive groups. It is expected that, similarly to the type $A$ case, the multiplicity question for these Hecke algebras can be resolved in terms of Kazhdan-Lusztig-Vogan polynomials.
\end{remark}

\subsection{Specializations of the parameters}\label{s:special}
The affine Hecke algebra $\CH$ with three parameters $(\bbq_0,\bbq_1,\bbq_2)$ considered in the previous discussion of Kato's exotic geometry gives rise to the following specializations.

The most general Hecke algebra corresponding to the root datum of type $B_n$ in the Bernstein presentation or, equivalently, to the labeled affine diagram of type $C_n$ in the Iwahori-Matsumoto presentation \cite{IM}
\begin{equation}
\xymatrix{q^{\lambda_0}\ar@{=>}[r]&q^{\lambda_1}\ar@{-}[r]&q^{\lambda_1}\ar@{-}[r]&\dotsb&q^{\lambda_1}\ar@{-}[l]&q^{\lambda_n}\ar@{=>}[l]}.
\end{equation}
is obtained by specializing $(\bbq_0,\bbq_1,\bbq_2)=(-q^{(\lambda_n-\lambda_0)/2},q^{(\lambda_n+\lambda_0)/2},q^{\lambda_1}).$ We recall that in the Iwahori-Matsumoto presentation, this algebra $\CH$ is generated by $\{T_0,T_1,\dots,T_n\}$, where $T_0$ corresponds to the affine simple reflection, subject to the braid relations in the affine Weyl group and the quadratic relations:
\begin{equation}
(T_0+1)(T_0-q^{\lambda_0})=0,\quad (T_n+1)(T_n-q^{\lambda_n})=0,\quad (T_i+1)(T_i-q^{\lambda_1})=0, \  1\le i\le n-1.
\end{equation}
For example, when $\lambda_0=1$, $\lambda_1=2$, and $\lambda_n=2$, we get the Iwahori-Hecke algebra of the special unitary group of a quaternion form in $2n$ variables, denoted $^2C_{2n}$ in \cite[page 64]{Ti}.

The classical Hecke algebras are obtained as specializations as follows. The affine Hecke algebra defined in terms of the root datum of $SO(2n+1,\bC)$ with two parameters, equivalently for the labeled affine Dynkin diagram in the Iwahori-Matsumoto presentation
\begin{equation}
\xymatrix{q^m&q\ar@{<=}[l]\ar@{-}[r]&q\ar@{-}[r]&\dotsb&q\ar@{-}[l]&q^m\ar@{=>}[l]}
\end{equation}
is obtained by setting $(\bbq_0,\bbq_1,\bbq_2)=(-1,q^m,q).$ For example, when $m=1$, we get the Iwahori-Hecke algebra of the symplectic group $Sp(2n,F)$.

The affine Hecke algebra for the root datum of $Sp(2n,\bC)$ with two parameters, equivalently for the labeled affine Dynkin diagram in the Iwahori-Matsumoto presentation
\begin{equation}
\xymatrix{q\ar@{-}[rd]\ar@/_/@{<.>}[dd]\\&q\ar@{-}[r]&q\ar@{-}[r]&\dotsb&q\ar@{-}[l]&q^m\ar@{=>}[l]\\q\ar@{-}[ru]}.
\end{equation}
is isomorphic to the one for the diagram
\begin{equation}
\xymatrix{1&q\ar@{<=}[l]\ar@{-}[r]&q\ar@{-}[r]&\dotsb&q\ar@{-}[l]&q^m\ar@{=>}[l]}
\end{equation}
and this is obtained by setting $(\bbq_0,\bbq_1,\bbq_2)=(-q^{m/2},q^{m/2},q).$ When $m=1$, this is the Iwahori-Hecke algebra of the split orthogonal $p$-adic group $SO(2n+1,F)$.

The Iwahori-Hecke algebra $\CH(SO(2n))$ of the special orthogonal group $SO(2n,F)$ (equivalently, defined in terms of the root datum of $SO(2n,\bC)$) corresponds to the Iwahori-Matsumoto diagram
\begin{equation}
\xymatrix{q\ar@{-}[rd]\ar@/_/@{<.>}[dd] &&&&&q\ar@{-}[ld]
\\&q\ar@{-}[r]&q\ar@{-}[r]&\dotsb&q\ar@{-}[l]
\\q\ar@{-}[ru]&&&&&q\ar@{-}[lu]}
\end{equation}
and it is isomorphic to the one for the diagram
\begin{equation}\label{e:O2n}
\xymatrix{&&&&&q\ar@{-}[ld]\\1&q\ar@{<=}[l]\ar@{-}[r]&q\ar@{-}[r]&\dotsb&q\ar@{-}[l]
\\&&&&&q\ar@{-}[lu]}
\end{equation}
Consider the specialization $\CH_B(0)$, $(\bbq_0,\bbq_1,\bbq_2)=(-1,1,q)$, which corresponds to the Iwahori-Matsumoto diagram
\begin{equation}\label{e:H_B0}
\xymatrix{1&q\ar@{<=}[l]\ar@{-}[r]&q\ar@{-}[r]&\dotsb&q\ar@{-}[l]&1\ar@{=>}[l]}.
\end{equation}
The algebra $\CH(SO(2n))$ has an outer automorphism $\delta$ defined on the Iwahori-Matsumoto generators by \[\delta(T_i)=T_i, \text{ if }0\le i\le n-2, \text{ and }\delta(T_{n-1})=T_n,\]
i.e., $\delta$ flips the branched nodes in (\ref{e:O2n}). Define the Hecke algebra of $O(2n)$ to be
\[\CH(O(2n))=\CH(SO(2n))\rtimes\<\delta\>.\] Then
\begin{equation}
\CH(O(2n))\cong \CH_B(0);
\end{equation}
more precisely, if the generators of $\CH_B(0)$ are $\{T_i': 0\le i\le n\}$ satisfying the Iwahori-Matsumoto relations for the diagram (\ref{e:H_B0}), then the algebra isomorphism is given by
\[T_i\mapsto T_i',\ 0\le i\le n-1, \text{ and }\delta\mapsto T_n'.
\]
 Thus the representation theory of $\CH(SO(2n))$ can be deduced from that of $\CH_B(0)$ via Clifford theory.

\smallskip

Notice that $\CH_A\subset \CH(SO(2n))\subset \CH_B(0)$. For every $\CH_A$-module $\pi$, form
\[I_D(\pi)=\CH(SO(2n))\otimes_{\CH_A}\pi\text{ and } I_B(\pi)=\CH_B(0)\otimes_{\CH_A}\pi=\CH_B(0)\otimes_{\CH(SO(2n))} I_D(\pi).
\]
\begin{lemma}\label{l:typeD}
Let $\pi$ be a simple $\CH_A$-module with separated central character with respect to $H_B(0)$, i.e., if the central character of $\pi$ is $S_n\cdot (q^{\lambda_1},\dots, q^{\lambda_n})$, then $q^{\lambda_i}\notin\{\pm 1\}$ for any $i$. Then $I_D(\pi)$ is reducible if and only if $I_B(\pi)$ is reducible.
\end{lemma}
\begin{proof}
It is clear that if $I_D(\pi)$ is reducible, then so is $I_B(\pi)$. For the converse, suppose $I_D(\pi)$ is irreducible. The central character of $I_D(\pi)$ is $W(D_n)\cdot (q^{\lambda_1},\dots, q^{\lambda_n})$, where $W(D_n)$ is the Weyl group of type $D_n$. We recall from Clifford theory (see for example \cite[Appendix]{RR}), that if $L$ is  a simple module of $\CH(SO(2n))$ then $\CH_B(0)\otimes_{\CH(SO(2n))}L$ is irreducible if and only if $^\delta L\cong L$ as $\CH(SO(2n))$-modules. Here $^\delta L$ is the $\delta$-twist of $L$.

At the level of Weyl groups $W_n=W(D_n)\rtimes \<\delta\>$.  The involution $\delta$ acts on the central characters as:
\[\delta(W(D_n)\cdot (q^{\lambda_1},\dots, q^{\lambda_n}))=W(D_n)\cdot (q^{\lambda_1},\dots, q^{-\lambda_n}).
\]
Now, if $q^{\lambda_i}\notin\{\pm 1\}$ for any $i$, then $(q^{\lambda_1},\dots, q^{\lambda_n})$ and $(q^{\lambda_1},\dots, q^{-\lambda_n})$ are not conjugate under $W(D_n)$ (only under $W(B_n)$). This means that $L$ and $^\delta L$ have different central characters, so they are nonisomorphic $\CH(SO(2n))$-modules. The claim is thus proved.
\end{proof}

\subsection{Arbitrary central characters}\label{s:general}
Let $\CR=(X,X^\vee,R, R^\vee,\Delta)$ be a general root datum as in section \ref{s:red}, specialize $\bbq=q>1,$ and let $\CH=\CH(\CR,\lambda,\lambda^*)$ be the generic affine Hecke algebra associated to this root datum and to parameter functions $\lambda,\lambda^*$. Specialize  $\bbq=q>1$. As in section \ref{s:red}, the center of $\CH$ is $\CA^W$, and the central characters are $W$-orbits of elements $s\in T=X^\vee\otimes_\bZ \bC^\times.$ The central character $W\cdot s$ is {\it real positive} if $s$ is hyperbolic. This definition is compatible with the one previously stated in the case of the Hecke algebra arising from Kato's geometry.

As it is well known, the classification of simple modules as well as reducibility questions regarding induced modules can be reduced to the case of real positive central characters, see for example \cite{Lu}, also \cite{BM}, \cite{OS}, or \cite[section 1]{CK}. Firstly, we recall this reduction process in the particular case of $\CH_B$, the affine Hecke algebra with unequal parameters for the root datum $(X,R,X^\vee,R^\vee,\Delta)$ for $SO(2n+1)$.


Fix $s=s_e s_h\in T$. The reflection $r_\alpha:X^\vee\to X^\vee$ with respect to the root $\alpha\in R$ defines an automorphism $r_\alpha:T\to T$. Define
\[R(s)=\{\alpha\in R: r_\alpha(s_e)=s_e\}.
\]
Let $\Delta(s)$ be the basis of $R_0(s)\cap R^+$.
Let $W(s)$ be the reflection subgroup of $W_n$ generated by the reflections with respect to $\Delta(s)$. Since $s$ lies in the simply-connected group $G=Sp(2n,\Bbb C)$, $W(s)$ equals the isotropy group of $s_e$ in $W_n$.
This simplifies the general reduction theorem, since no outer automorphisms of the Dynkin diagram $\Delta(s)$ appear.

Let $q>1$, $m_1,m_2\in\bR$ and $a=(s,-q^{m_1}, q^{m_2},q)$ be a semisimple element of $\CG$. Consider the (possibly non-semisimple) root datum $\CR_s=(X,R(s),X^\vee,R(s)^\vee,\Delta(s))$. As in the Bernstein presentation, define a parameter function $\lambda_s: \Delta(s)\to \bR$ as follows: for every $\epsilon_i\pm\epsilon_j\in \Delta(s)$, set $\lambda_s(\epsilon_i\pm\epsilon_j)=1$, and for every $\epsilon_i\in \Delta(s)$, set
\begin{equation}\label{e:spec}
\lambda_s(\epsilon_i)=(m_2+m_1)/2+\epsilon_i(s_e) (m_2-m_1)/2;
\end{equation}
notice that if $\epsilon_i\in \Delta(s)$, then $\epsilon_i(s_e)\in\{\pm 1\}.$ Set also $\lambda^*_s(\alpha)=\lambda_s(\alpha)$ for all $\alpha\in \Delta(s)$ such that $\alpha^\vee\in 2X^\vee.$

\

Let $\CH(s)=\CH(\CR_s,\lambda_s,\lambda_s^*)$ be the affine Hecke algebra defined as in Definition \ref{d:bernstein}.The following theorem is a particular case of the reduction theorems of Lusztig \cite[Theorems 8.6 and 9.3]{Lu}, see also \cite[Theorem 2.6 and Corollary 2.10]{OS} or \cite[section 1.2]{CK}.

\begin{theorem}\label{t:reduction}
Let $q>1$, $m_1,m_2\in\bR$ and $a=(s,-q^{m_1}, q^{m_2},q)$ be a semisimple element of $\CG$ with $s=s_es_h$. There exists an equivalence of categories between $\CH_a$-mod and $\CH(s)_{s_h}$-mod, where $\CH(s)$ is the affine Hecke algebra defined in terms of the root datum $\CR_s$ and parameter functions $\lambda_s,\lambda_s^*$ as above.\footnote{The more customary reduction theorem is phrased in terms of the graded affine Hecke algebra \cite{Lu}. Since we will not use the graded Hecke algebra anywhere else in the paper, we chose to phrase the reduction theorem in this way.}
\end{theorem}

\begin{remark}\label{r:induction-compatible} The equivalence of categories in Theorem \ref{t:reduction} is compatible with the parabolic induction, see for example \cite[Theorem 6.2]{BM}.
\end{remark}

\begin{example}\label{e:example centralizer}
An important example is the following. Assume $s_e=(-1,\dots,-1,1,\dots,1)\in T$ is such that $G(s_e)=Sp(2n_1)\times Sp(2n_2)$ with $n_1+n_2=n$. Suppose $s=s_1s_2\in T$, $s_i\in Sp(2n_i)$, $i=1,2$, has elliptic part $s_e$. Set $a=(s,-q^{m_-}, q^{m_+},q)$. Then $(\CH_{B_n})_a$-mod is equivalent with $(\CH_{B_{n_1}})_{\bar a_1}\times (\CH_{B_{n_2}})_{\bar a_2}$-mod, where $\bar a_1=(s_{1,h},-1,q^{m_-},q)$ and $\bar a_2=(s_{2,h},1,q^{m_+},q)$.
\end{example}

\begin{remark}\label{r:real-transfer}
Theorem \ref{t:reduction} is proved by exhibiting an explicit Morita equivalence between $\CH_a$ and $\CH(s)_{s_h}$. More precisely, there is an algebra isomorphism $\overline\CH_a\cong \textup{Mat}_k(\overline\CH(s)_{s_h})$, where $\textup{Mat}_k$ denotes the $k\times k$-matrix algebra, $\overline\CH_a$ and $\overline\CH(s)_{s_h}$ are certain completions of the algebras, and $k=|G/G(s_e)|.$ See \cite[Theorems 8.6, 9.3]{Lu}, also \cite[Theorems 2.6, 2.8]{OS}, for the details. Because of this, most representation theoretic questions, in particular, our reducibility question, can be transferred to the setting of real positive central character.

\end{remark}

\begin{remark}\label{r:real-gl}
The reduction to real central character for the affine Hecke algebra $\CH_A$ of $GL(n)$ is much simpler. If $S_n\cdot s$ denotes a central character, $s=s_e s_h$, then there is a equivalence of categories between $(\CH_A)_s$-mod and $(\CH_A(s))_{s_h}$-mod, where $\CH_A(s)$ is the affine Hecke algebra for the centralizer of $s$ in $GL(n,\bC)$, which is again a product of general linear groups.
\end{remark}

Suppose $\psi:G_1\twoheadrightarrow G_2$ is an isogeny of the complex reductive groups $G_1$ and $G_2$. Let $\psi^*:\CR_2\to \CR_1$ be the induced isogeny of root data. Let $\CH_i=\CH(\CR_1,\lambda_i,\lambda_i^*)$, $i=1,2$ be affine Hecke algebras with parameters as in Definition \ref{d:bernstein}. If the parameters are compatible with respect to $\psi^*$, which we assume now,  then we have an injective algebra homomorphism $\psi^*: \CH_2\to \CH_1$, see for example \cite[sections 1.4, 1.5]{Re}. An immediate consequence of the reduction to real central character theorems \cite[Theorems 8.6 and 9.3]{Lu} is the following:

\begin{lemma}\label{l:isogeny} The algebra homomorphism $\psi^*$ induces an equivalence of categories between the modules with real positive central characters of $\CH_1$ and $\CH_2$.
\end{lemma}
In other words, the classification of modules with real positive central character for an affine Hecke algebra is independent of the isogeny of the root datum.

\begin{proposition}\label{p:true}
Let $\CH$ be an affine Hecke algebra for a root datum of type $GL(n)$, $B_n$, $C_n$, or $D_n$ (of arbitrary isogeny) with unequal parameters, or $\CH=\CH(O(2n))$. Let $\CH_A$ be the $GL$-Hecke subalgebra of $\CH$ and let $\pi$ be a simple $\CH_A$-module with real positive separated central character. Then the reducibility of $I(\pi)$ and the structure of its lattice of submodules are independent of the parameters of $\CH$.
\end{proposition}

\begin{proof}
When $\CH$ is an affine Hecke algebra for a general linear group, the proposition is a tautology.

Otherwise, via the specializations presented in the previous subsection and Lemma \ref{l:typeD}, one is, for the cases therein, reduced to consider the case of the affine Hecke algebra $\CH $ for the root datum of $SO(2n+1)$ with unequal parameters and of a semisimple element $s$ with $s_e=1$. Then, example \ref{e:example centralizer} applied to $s_e=1$ brings us back to the situation considered in the subsection \ref{s:real positive} and one concludes by Corollary \ref{c:conj-true}.

Finally, because of the independence of isogeny at real positive central character Lemma \ref{l:isogeny}, the claim holds for all root data.
\end{proof}

\begin{corollary}\label{c:true}
Let $\CH_B$ be the affine Hecke algebra for the root datum of $SO(2n+1,\bC)$ and unequal parameters. Let $\pi$ be a simple $\CH_A$-module with separated central character. Then the reducibility of $I(\pi)$ and the structure of its lattice of submodules are independent of the parameters of $\CH_B$.
\end{corollary}

\begin{proof}Suppose that the central character of $\pi$ is represented by the semisimple element $s$. Consider the algebras $\CH(s)$ from Theorem \ref{t:reduction} and $\CH_A(s)$ from Remark \ref{r:real-gl}. By \cite[Theorem 6.2]{BM}, the reduction to real positive central character commutes with parabolic induction. Hence if $\pi'$ is the $(\CH_A(s))_{s_h}$-module corresponding to $\pi$, then $I(\pi)$ is reducible if and only if $I(\pi')=\CH(s)\otimes_{\CH_A(s)}\pi'$ is. Moreover, as it can be seen from the reduction to real positive central character (Theorem \ref{t:reduction} and (\ref{e:spec})), the separatedness condition for the central character as defined in section \ref{s:3.5} corresponds to the definition of separated real positive central character (Definition \ref{d:exotic-separated}), see Remark \ref{r:compatible}.

The algebra $\CH(s)$ corresponds to the root datum $\CR_s$ which is a product of the root data of type $GL$, $B$, $C$, $D$, while $\CH_A(s)$ corresponds to the maximal root subsystem of type $GL$ of $\CR_s$. Then Proposition \ref{p:true} implies the claim of the corollary.
\end{proof}

\begin{corollary}\label{c:true2}
Let $\CH$ and $\CH'$ be two affine Hecke algebras for root data of type $SO(2n+1,\bC)$, $Sp(2n,\bC)$, $SO(2n,\bC)$, or for $O(2n,\bC)$, with arbitrary parameters. Let $\CH_A$ be the $GL$-Hecke algebra viewed as a subalgebra of $\CH$ and $\CH'$. Let $\pi$ be a simple $\CH_A$-module with separated central character with respect to both $\CH$ and $\CH'$, and let $I^\CH(\pi)$ and $I^{\CH'}(\pi)$ be the corresponding induced modules. Then there exists an isomorphism of the lattices of submodules for $I^{\CH}(\pi)$ and $I^{\CH'}(\pi)$, in particular, $I^{\CH}(\pi)$ is irreducible if and only if $I^{\CH'}(\pi)$ is.
\end{corollary}

\begin{proof}
This follows now from Corollary \ref{c:true} via the specializations in the previous subsection, similar as in the proof of Proposition \ref{p:true}.
\end{proof}

\subsection{}The results of section \ref{s:red}, also imply the following complementary result.

\begin{proposition}\label{p:reduc}
Let $\CH$ be an affine Hecke algebra for a root datum of type $SO(2n+1,\bC)$, $Sp(2n,\bC)$, or for $O(2n,\bC)$ with arbitrary parameters. If the central character of the $\CH_A$-module $\pi$ is not separated with respect to the parameters of $\CH$, then $I(\pi)$ is reducible.
\end{proposition}

\begin{proof} By Theorems \ref{t:reducible} and \ref{t:reducible2} and Remark \ref{r:red}, the claim holds when $\CH$ is the affine Hecke algebra for the root datum of $SO(2n+1,\bC)$ with any three parameters. By the specializations in  subsection \ref{s:special}, it also holds when $\CH$ has the root datum of $Sp(2n,\bC)$ with two parameters or $O(2n,\bC)$ with one parameter.

\end{proof}

\section{Consequences for representations of $p$-adic groups}
Recall that we fixed a type $G$ of quasi-split classical groups (either a general linear, a symplectic, a special odd orthogonal, a (full) even orthogonal or a unitary group) and also an extension field $E$ of $F$, so that the Levis subgroups of groups of type $G$ have the form $M=M'\times G_{d_0}$, where $G_{d_0}$ denotes a group of type $G$ of relative semi-simple rank $d$ and $M'$ a product of general linear groups with coefficients in $E$.

\begin{definition} For $d_0$ an integer $\geq 0$ and $\sigma $ an irreducible supercuspidal representation of $G_{d_0}$, we denote by $S_{\sigma }$ the set of irreducible supercuspidal representations $\rho $ of general linear groups, so that the representation obtained by (normalized) parabolic induction from $\rho\otimes\sigma $ is reducible. (More precisely, if $\rho $ is a representation of $GL_k(E)$, then $\rho\otimes\sigma $ is a representation of the standard Levi subgroup $GL_k(E)\times G_{d_0}$ of $G_{d_0+k}$, and $\rho\in S_{\sigma }$ means that the representation of $G_{d_0+k}$ obtained by (normalized) parabolic induction from $\rho\otimes\sigma $ is reducible.)

If $\pi $ is an irreducible smooth complex representations of some $GL_k(E)$, we will say that the supercuspidal support of $\pi $ is \rm separated \it from $\sigma $, if no element of $S_{\sigma }$ is a factor of an element in the supercuspidal support of $\pi $.
\end{definition}

\begin{definition} Let $\CH $ be an extended affine Hecke algebras with parameters, which is a tensor product of extended affine Hecke algebras $\CH _i$ of root data of type $GL$, $SO$ or $Sp$. Let $M$ be a simple $\CH _A$-module, which means that $M$ is a tensor product of simple $\CH _{i,A}$-modules $M_i$.

Then, we say that $M$ has separated support relative to $\CH $, if and only if each $M_i$ has separated support relative to $\CH_i$.
\end{definition}

If $\CO $ is the inertial orbit of an irreducible supercuspidal representation of a standard Levi subgroup of a general linear group and $\sigma $ an irreducible supercuspidal representation of a group of type $G$, let us denote by $\CH_\CO $ and $\CH _{\CO\otimes\sigma}$ the algebras denoted in section \ref{s:Bernstein} respectively by $End_{GL_{d(\CO)}}(\CP _{\CO })$ and $End_{G_{d(\CO\otimes\sigma)}}(\CP _{\CO\otimes\sigma })$. If $(\pi ,V)$ is an object in $Rep(GL_{d(\so )})_{\CO }$, then we denote by $M_{\pi }$ the corresponding $\CH_{\CO }$-module $Hom_{GL_{d(\CO )}}(\CP _{\CO},V)$.

\begin{proposition}\label{p:4.16} Let $\CO $ be the inertial orbit of an irreducible supercuspidal representation of a standard Levi subgroup $M'$ of a general linear group $GL_k(E)$, $\sigma $ an irreducible supercuspidal representation of $G_{d_0}$ and $d=k+d_0$. Write $M=GL_k(E)\times G_{d_0}$. Let $(\pi ,V)$ be a smooth irreducible representation in $Rep(GL_k(E))_{M',\CO}$. Then, the supercuspidal support of $\pi $ is separated from $\sigma $, if and only if the $\CH_{\CO }$-module $M_{\pi }$ has separated support relative to $\CH_{\CO\otimes\sigma}$.
\end{proposition}

\begin{proof} Write $\CO =\CO _1\otimes\cdots\otimes\CO_r$, where the $\CO_i$ are disjoint homogeneous inertial orbits. Then $\CH_{\CO\otimes\sigma }$ is the tensor product of the $\CH_{\CO_i\otimes\sigma }$ and $M_\pi $ is a tensor product of $\CH_{\CO_i\otimes\sigma }$-modules $M_i$.

We have to show that the supercuspidal support of $\pi $ is separated from $\sigma $, if and only if each $M_i$ has separated support relative to $\CH_{\CO_i\otimes\sigma }$.

By \cite[Corollary 3.4]{H2}, there is a map $\CO_i\rightarrow\Bbb C^\times$, $\rho\mapsto \lambda_{\rho }$, such that, for $\rho $ in $\CO_i$, $\rho $ is factor of a representation in the supercuspidal support of $\pi $, if and only if $\lambda _{\rho }$ is a coefficient of a weight of $M_i$. In addition, as the equivalence of categories preserves parabolic induction, the representation parabolically induced from $\rho\otimes\sigma $ is irreducible, if and only if the $\CH_{\CO_\rho \otimes\sigma }$-module $\CH_{\CO_\rho \otimes\sigma }\otimes _{\CH_{\CO_\rho }}M_\rho $ is irreducible (where $\CO_{\rho }$ denotes the inertial orbit of $\rho $). The latter is by the Lemmas \ref{l:rank-one1} and \ref{l:rank-one2} equivalent to say that $M_{\rho }$ has separated support relative to $\CH_{\CO_\rho \otimes\sigma }$. Saying that this is true for each factor of a supercuspidal representation in the supercuspidal support of $\pi $, is equivalent to say that the weight vector of each $M_i$ satisfies the separateness condition relative to $\CH_{\CO_i\otimes\sigma}$. This proves the proposition.
\end{proof}

We can now prove the main theorem of the paper, which was conjectured in \cite{LT}:

\begin{theorem}\label{c:4.16}
Let $\pi $ be an irreducible representation of $GL_k(E)$. Assume $E=F$ (resp. $E\ne F$). Let $G$ be a type of classical groups defined over $F$ that is not a unitary group (resp. that is a unitary groups with splitting field $E$), let $d_0$ be a nonnegative integer and let $\sigma$ be an irreducible supercuspidal representation $\sigma $ of $G_{d_0}$ which is separated from the supercuspidal support of $\pi $. 
Then the structure of the lattice of submodules and, in particular, the reducibility of the representation of $G_{d_0+k}$ obtained by (normalized) parabolic induction from $\pi\otimes\sigma $ are independent of $G$, $d_0$, and $\sigma$ (i.e., they only depend on $\pi$).
\end{theorem}

\begin{proof} Denote by $\CO$ the supercuspidal support of $\pi $. By Theorem \ref{t:2.5} and Remark \ref{r:2.6}, the Bernstein component $Rep(G_{d(\CO)})_{\CO\otimes\sigma }$ is equivalent to the category of right modules over a tensor product of extended affine Hecke algebras of type $A$, $B$ or $D$ as above and this equivalence of category is compatible with parabolic induction and preserves the lattice of submodules of corresponding objects by Lemma \ref{l:4.10}. Saying that $\pi $ satisfies the support condition relative to $\sigma $ is by Proposition \ref{p:4.16} equivalent to say that the corresponding modules over the extended affine Hecke algebras of type $A$, $B$ and $D$ satisfy this support condition. It is now immediate from Corollary \ref{c:true}, \ref{c:true2}  that the structure of the lattice of submodules of the representation induced from $\pi\otimes\sigma $ and, in particular, its reducibility do not depend on the parameters, even if one switches from an affine Hecke algebra of type $B$ to an extended affine Hecke algebra of type $D$ (or inversely - this may actually happen when changing $\sigma $, as \cite[Proposition 1.13, 6.1] {H1} shows with \cite[A.2, C.5]{H3}, or when changing the type of $G$ according to our statement).
\end{proof}

\begin{remark}\label{r:reducibility} (i) By arguments analogous to the ones used in the proof of Theorem \ref{c:4.16}, it follows from Proposition \ref{p:reduc} that the representation parabolically induced from $\pi\otimes\sigma $ is always reducible, if the support condition is not satisfied. This gives another proof of \cite[Theorem 1.1]{LT}, which is also valid in positive characteristic.

(ii) The Theorem \ref{c:4.16} and the preceding remark become in general wrong, if one switches from a non-unitary group to a unitary group (or inversely): this comes from the fact that the appearance (or not) of affine Hecke algebras of type A is related to the property of a representation of a general linear group to be respectively \rm self-dual \it or \rm self-(conjugated-)dual. \it Of course, these properties are in general not equivalent.
\end{remark}

\end{document}